\documentclass[a4paper]{amsart}
\usepackage{pentathm}
\usepackage{pgf}
\usepackage{pgfpages}
\usepackage{pentacmds}
\title{Explicit associator relations for multiple zeta values}
\author{Ismael Soud{\`e}res}
\thanks{I want to thank Leila Schneps, Pierre Cartier
  for their support, their kind attention to my work and all the
discussions that lead to this article. Benjamin
Collas was a great help in mastering the computer, relations
in the annexe being computed using Mathematica. I also want to thanks
particularly  Francis
Brown for all his important and useful comments.}
\address{Fachbereich Mathematik \\
Universität Duisburg-Essen, Campus Essen \\
Universitätsstrasse 2\\
45117 Essen\\
Germany \\
ismael.souderes@uni-due.de}
\date{\today}

\begin{document}
\maketitle
\begin{abstract} Associators were introduced by Drinfel'd in \cite{DrinQTQH} as
  a monodromy  
representation of a Knizhnik-Zamolodchikov equation.  Associators can be
briefly described as formal series in two non-commutative variables
satisfying three equations.  These three equations yield a large number
of algebraic relations between the coefficients of the series, a situation
which is particularly interesting in the case of the original Drinfel'd 
associator, whose coefficients are multiple zetas values.  In the first
part of this paper, we work out these algebraic relations among multiple
zeta values by direct use of the defining relations of associators. While 
well-known for the first two relations, the algebraic relations we obtain
for the third (pentagonal) relation, which are algorithmically explicit
although we do not have a closed formula, do not seem to have been 
previously written down.  The second part of the paper shows that if one has
an explicit basis for the bar-construction of the moduli space $\m_{0,5}$ of 
genus zero Riemann surfaces with $5$ marked points at one's disposal, then
the task of writing down the algebraic relations corresponding to the pentagon 
relation becomes significantly easier and more economical compared to the 
direct calculation above.
We discuss the explicit basis described by Brown, Gangl and Levin, which is dual to
the basis of the enveloping algebra of the braids Lie algebra $\UB$.

In order to write down the relation between multiple zeta values, we
then remark that it is enough to write down the relations associated to
elements that generate the bar construction as an algebra. This
corresponds to looking at the bar construction modulo shuffle, which is dual
to the Lie algebra of $5$-strand braids. We write down, in the
appendix, the associated algebraic relations between 
multiple zeta values in weights $2$ and $3$.
\end{abstract}

\tableofcontents

\section{Introduction}

In the first part of this introduction we recall the necessary
definitions concerning associators, and in the second part, we recall
the definitions and main results concerning multiple zeta values.  In the
third part, we give the outline of the paper and state the main results.
\subsection{Associators}
Let $k$ be a field of characteristic $0$. Let $\UF=k\llan X_0,X_1\rran$
be the ring of formal power series over $k$ in two non-commutative
variables. The coproduct $\Delta$ on $\UF$ is defined by 
\[
\Delta(X_0)=X_0 \ot 1 + 1 \ot X_0 \qquad \Delta(X_1)=X_1 \ot 1 + 1 \ot X_1.
\]
An element $\Phi= \Phi(X_0, X_1) \in \UF$ is \emph{group-like} if it
satisfies $\Delta(\Phi)=\Phi \hat{\ot}\Phi$ where $\hat{\ot}$ denotes
the complete tensor product. 
\begin{rem}We remark that the constant term of a
group-like element is $1$.
\end{rem}

\begin{defn}
If $S$ is a finite set, let $S^*$ denote the set of words with letters
in $S$, that is the dictionary over $S$. If $S=\{s_1,\ldots, s_n\}$
we may write $\{s_1,\ldots, s_n\}^*$.

Let $\ax$ be the dictionary over $\{X_0, X_1\}$.
\end{defn}
We remark that the monomials in $\UF$ are words in $\ax$; the empty word 
$\emptyset$ in $\ax$ will be $1$ by convention when considered in $\UF$. 
The following definition allows us to define a filtration on $\UF$.
\begin{defn}
The depth $\de(W)$ of a monomial   $W\in\UF$, that is an element of
$\ax$, is the number of $X_1$'s, and
its weight (or length) $\we(W)=|W|$ is the number of letters.
\end{defn}
The algebra $\UF$ is filtered by the weight, and its graded pieces of
weight $d$ are the subspaces generated by the monomials of length $d$;
$\UF$ is thus a graded algebra. 

Let $\UB$ be the enveloping algebra of $\mf B_5$, the completion (with respect 
to the natural grading) of the pure sphere braid Lie algebra \cite{IharaAPSBG}; that is,
$\UB$ is the quotient of $k\llan X_{ij}\rran $ with $1\leqs i \leqs
5$ and $1\leqs j \leqs 5$ by the relations
\begin{itemize}
\item $X_{ii}=0$ for $1 \leqs i \leqs 5$,
\item $X_{ij}=X_{ji}$ for $1 \leqs i,j \leqs 5$,
\item $\sum\limits_{j=1}^5 X_{ij}=0$ for $1 \leqs i \leqs 5$, 
\item $[X_{ij}, X_{kl}]=0$ if $\{i,j\}\cap\{k,l\}=\emptyset$.
\end{itemize} 

\begin{defn}[{Drinfel'd \cite{DrinQTQH}}]\label{ass}
A group-like element $\Phi$ in $\UF$ having coefficients equal to zero 
in degree $1$, together with an element $\mu \in k^*$, is 
an associator if it satisfies the following equations
\begin{gather}
\label{rel1}  \Phi(X_0,X_1)\Phi(X_1,X_0)=1 \tag{I}, \\
\label{rel2} \tag{II}
e^{\frac{\mu}{2}X_0}\Phi(X_{\infty},X_0)e^{\frac{\mu}{2} X_{\infty}}
\Phi(X_1,X_{\infty})e^{\frac{\mu}{2}X_1}\Phi(X_0,X_1)=1, 
\intertext{with $X_0+X_1+X_{\infty}=0$, and } 
\tag{III} \label{rel3}
\Phi(X_{12},X_{23}) \Phi(X_{34},X_{45}) \Phi(X_{51},X_{12})
\Phi(X_{23},X_{34}) \Phi(X_{45},X_{51})=1, 
\end{gather} 
where \eqref{rel3} takes place in $\UB$.

We will write an associator as
 \[
\Phi(X_0,X_1)=\sum_{W\in \ax}  Z_W W=1+\sum_{\substack{ W\in \ax\\ W\neq 
\emptyset}} Z_W W.
\] 
We have $Z_{\emptyset}=1$ because $\Phi$ is group-like.
\end{defn}
In \cite{DrinQTQH}, Drinfel'd gives an explicit associator $\pkz$ over $\C$, 
known as the Drinfel'd associator and associated to a
Knizhnik-Zamolodchikov equation (KZ equation). More precisely,
consider the KZ equation (one can also see \cite{FurushoMZVSDA}[\S 3]).
\begin{gather}
\tag{KZ} \frac{\partial g}{\partial u} = 
\left(\frac{X_0}{u}+\frac{X_1}{u-1}
\right)\cdot g(u) \label{KZ}
\end{gather}
where $g$ is an analytic function in one complex variable $u$ with
values in $\C\llan X_0,X_1\rran$ (analytic means that each coefficient
is an analytic function). This equation has singularities only
at $0$, $1$ and $\infty$. The equation \eqref{KZ} has a unique
solution on $C=\C\sm(]-\infty, 0] \cup [1,\infty[)$ having a specified
value at a given point in $C$, because $C$ is simply
connected.  Moreover, at $0$ (resp. $1$), there exists a unique solution
$g_0(u)$ (resp. $g_1(u)$)  such that 
\[
g_0(u)\sim u^{X_0} \quad (u \ra 0)
\qquad \left(\mx{resp. }
g_1(u)\sim (1-u)^{X_1} \quad (u \ra 1) \right).
\]
As $g_0$ and $g_1$ are invertible with specified asymptotic behavior,
they must coincide up to multiplication on the right by an
invertible element in $\C\llan X_0,X_1\rran$.
\begin{defn}
The \emph{Drinfel'd associator}
\footnote{In \cite{DrinQTQH}, Drinfel'd actually defined $\phi_{KZ}$
  rather than $\pkz$, where 
$\phi_{KZ}(X_0,X_1)=  \pkz(\frac{1}{2i\pi} X_0, \frac{1}{2i\pi}X_1)$ and is 
defined via the KZ equation $ \frac{\partial g}{\partial u} = \frac{1}{2i\pi}
\left(\frac{X_0}{u}+\frac{X_1}{u-1} 
\right)\cdot g(u)$.}
  $\pkz$ is the element in $\C\llan X_0,X_1\rran $ defined by 
\[
g_0(u)= g_1(u)\pkz(X_0,X_1).
\]
\end{defn}  
In \cite{DrinQTQH}, Drinfel'd proved the following result.
\begin{prop} The element $\pkz$ is a group-like element and it
  satisfies \eqref{rel1}, \eqref{rel2} with $\mu=2i\pi$, and \eqref{rel3}
  of definition \ref{ass}.  That is, 
\begin{gather}
\label{I}  \pkz(X_0,X_1)\pkz(X_1,X_0)=1 \tag{I\tsub{KZ}} 
\end{gather}
\begin{multline}
e^{i\pi X_0}\pkz(X_{\infty},X_0)e^{i\pi X_{\infty}}
\pkz(X_1,X_{\infty})e^{i\pi X_1}\pkz(X_0,X_1)=1 
\\ \mx{with}\quad X_0+X_1+X_{\infty}=0\tag{II\tsub{KZ}} \label{II}
\end{multline}
\begin{multline}
\pkz(X_{12},X_{23}) \pkz(X_{34},X_{45}) \pkz(X_{51},X_{12})
\pkz(X_{23},X_{34}) \\
\pkz(X_{45},X_{51})=1 
 \qquad \mx{in } \UB\tag{III\tsub{KZ}} .\label{III}
\end{multline} 
\end{prop}

\subsection{Multiple zeta values}
For a $p$-tuple $\mathbf{k} = (k_1,\ldots , k_p)$ of strictly positive 
integers with $k_1 \geqs 2$, the  multiple zeta value $\zeta(\mathbf{k})$ is 
defined as
\[\zeta(\mathbf{k}) = \sum_{n_1>\ldots>n_p>0} \frac{1}{n_1^{k_1}\cdots
  n_p^{k_p}}.\] 
\begin{defn}
The depth of a $p$-tuple of integers $\mb k= (k_1, \ldots , k_p)$ is
$\de(\mb k)=p$, and its weight $\we(\mb k)$ is $ \we(\mb k)=k_1 +\cdots +k_p$.
\end{defn}

To the tuple of integers $\mathbf{k}$, with $n=\we(\mb k)$, we associate the
$n$-tuple $\ol k$ of $0$ and $1$ by:
\[\ol{k}=(\underbrace{0,\ldots , 0}_{k_1-1\mx{ times}}, 1,\ldots
  ,\underbrace{0,\ldots ,0}_{k_p-1\mx{ times}},
  1)=(\ve_n,\ldots,\ve_1)\]
and the word in $\{X_0, X_1\}^*$ 
\[
X_{\ve_n}\cdots X_{\ve_1}.
\]
This makes it possible to associate a multiple zeta value $\zeta(W)$
to each word $W$ in $X_0 \{X_0,X_1\}^* X_1$ (where $W$ begins with $X_0$
and ends with $X_1$).

Following Kontsevich and Drinfel'd, one can write the multiple zeta values as a Chen
iterated integral \cite{IIDFLSChen}
\[
\zeta(\mb k )= \int_0^1 (-1)^p\frac{\du}{u- \ve_n} \circ \cdots \circ
\frac{\du}{u- \ve_1}. 
\]
Note that, as $k_1 \geqs 2 $, we have $\ve_n =0$. This expression as
an iterated integral leads directly to an expression of the multiple zeta values as
an integral over a simplex
\[
\zeta(\mb k )= \int_{\Delta_n}  (-1)^p\fr{\dt_1}{t_1-\ve_1}\w 
\cdots\w \fr{\dt_n}{t_n-\ve_n}
\]
where $\Delta_n=\{0 <t_1 <\ldots <t_n< 1\}$.

Thanks to the work of Boutet-de-Monvelle, Ecalle, Gonzales-Lorca and Zagier,
with the further developments by Ihara, Kaneko or Furusho,  we can
extend the definition of multiple zeta values to tuples without the condition $k_1\geqs
2$ (see \cite{GonLorca}, \cite{DMRRac}, \cite{IKZ} or \cite{FurushoMZVSDA}).  These
extended multiple zeta values are called regularized multiple zeta values, and we speak of
regularizations. We will be interested in a specific regularization, the
\emph{shuffle regularization}.
\begin{defn}[Shuffle product] A shuffle of $\{1,2, \ldots, n\}$ and
  $\{1, \ldots, m\}$ is a permutation $\sigma$ of $\{1,2,\ldots, n+m\}$ such
  that:
\[
\sigma(1)<\sigma(2)<\cdots <\sigma(n) 
\qquad \mx{and} \mx \qquad
\sigma(n+1)<\sigma(n+2)<\cdots <\sigma(n+m). 
\]
The set of all the shuffles of $\{1,2, \ldots, n\}$ and
  $\{1, \ldots, m\}$ is denoted by $\on{sh}(n,m)$

Let $V=X_{i_1}\cdots X_{i_n}$ and $W=X_{i_{n+1}}\cdots X_{i_{n+m}}$ be
two words in $\ax$. The shuffle of $V$ and $W$ is the
collection of words 
\[
\on{sh}(V, W) = (X_{i_{\sigma^{-1}(1)}}X_{i_{\sigma^{-1}(2)}} \cdots
X_{i_{\sigma^{-1}(n+m)}} )_{\sigma \in \on{sh}(n,m)}.
\]
Working in $\C\llan X_0,X_1\rran$, we will also consider the sum 
\[
V\sha W =\sum_{U\in \on{sh}(V, W) }U=
\sum_{\sigma \in \on{sh}(n,m)}
X_{i_{\sigma^{-1}(1)}}X_{i_{\sigma^{-1}(2)}} \cdots
X_{i_{\sigma^{-1}(n+m)}} 
\]
and extend the shuffle product $\sha$ by linearity.
\end{defn}

\begin{defn} The \emph{shuffle regularization} of the multiple zeta
  values is the 
  collection of real numbers $\big(\zsha(W)\big)_{W\in\ax}$
  such that:
\begin{enumerate}
\item $\zsha(X_0)=\zsha(X_1)=0$,
\item $\zsha(W)=\zeta(W)$ for all $W\in \xay $,
\item $\ds \zsha(V)\zsha(W)=\sum_{U\in \on{sh}(V,W)}\zsha(U)$ for all
  $V,W \in \ax$
\end{enumerate}
\end{defn}
These regularized multiple zeta values $\zsha(W)$, for $W$ not in $\xay$,
are in fact linear combinations of the usual multiple zeta values, which were
given explicitly by Furusho in \cite{FurushoMZVSDA}. Seeing
$\zsha$ as a linear map from $\C\llan X_0,X_1\rran$ to $\R$, one can then
rewrite the third condition as 
\[
\zsha(V\sha W)=\zsha(V)\zsha(W).
\]

The coefficients of the Drinfel'd associator can be written in an
explicit way using convergent multiple zeta values
\cite{FurushoMZVSDA}.
\begin{prop} Using the shuffle regularization we can write
(\cite{TTQMurakamiKIKP}, \cite{GonLorca}, \cite{FurushoMZVSDA})
\[
\pkz(X_0,X_1)=\sum_{W\in\ax} (-1)^{\de(W)}\zsha(W)W.
\]
\end{prop}
\subsection{Main results}
In Theorem \ref{rel1coef} and Theorem \ref{rel2coef} we will give
explicit relations between the 
coefficients of the series defining an associator $\Phi$ equivalent to
the relation \eqref{rel1} and \eqref{rel2} satisfied by $\Phi$. Both
were well-known, as it is easy to expand the product of the associators in
$\UF$, even if the author does not know whether the relations of Theorem
\ref{rel2coef} have actually appeared explicitly in the literature. In the
case of the pentagon relation \eqref{rel3}, writing down relations
between the coefficients implies fixing a basis $B$ of $\UB$. Even if
fixing such a basis breaks the natural symmetry of the pentagon
relation \eqref{rel3}, it makes it possible to give an explicit family of
relations between the coefficients of $\Phi$ equivalent to
\eqref{III}. More precisely, decomposing a word $W$ in the subset of letters
$X_{34},X_{45},X_{24},X_{12},X_{23} $ in the basis $B$ we have 
\[
W=\sum_{b\in B} l_{b,W} b,
\]
and we obtain the following theorem.
\begin{thma}[Theorem \ref{rel3coef}]
The relation \eqref{rel3} is equivalent to the family of relations
\begin{flalign*}
\forall b \in B\, \, (b\neq 1) & \qquad \qquad 
\sum_{W\in \{X_{34},X_{45},X_{24},X_{12},X_{23}\}^*} l_{b,W} C_{5,W} =0, & &
\end{flalign*}
where $C_{5,W}$ are explicitly given by: 
\[
C_{5,W}=\sum_{\substack{U_1,\ldots,U_5 \in \axb \\ U_1\cdots U_5=W}}
Z_{\pa(U_1)}Z_{\pb(U_2)}Z_{\pc(U_3)}Z_{\pd(U_4)}Z_{\pe(U_5)}.
\]
In the above formula, $\axb$ denotes $ \{X_{34},X_{45},X_{24},X_{12},X_{23}\}^*$
 and the $\pg$ are maps from $\UB$ to $\UF$ defined on the letters
$X_{12}$, $X_{23}$,  $X_{34}$, $X_{45}$, $X_{24}$ in
Definition \ref{pgdef} (as example: $\pa(X_{12})=X_0$, $\pa(X_{23})=X_1$ and
$\pa(X_{34})=\pa(X_{45})=\pa(X_{24})=0$) with the convention that $Z_0=0$.
\end{thma}
Applying this theorem to the particular basis $B_4$ coming from the
identification 
\[
\UB\simeq k\llan X_{34},X_{45},X_{24} \rran \rtimes k\llan
X_{12},X_{23}\rran,  
\]
one can compute the coefficients $l_{b,W}$ using the equation defining
$\UB$ (here $\rtimes$ denotes the complete semi-direct product). In particular it
is easy to see that $l_{b,W}$ is in $\Z$ in that case. As shown by Ihara in the
Lie algebra setting (\cite{IharaAPSBG}), the above identification
is induced by the morphism $ f_4 : \UB \lra \UF$ that sends $X_{i4}$ to $0$,
$X_{12}$ to $X_0$ and $X_{23}$ to $X_1$ and by a particular choice of generators
of the kernel (that is $X_{24}$, $X_{34}$ and $X_{45}$).

After explaining each family of relations between the coefficients, we apply our
results to the particular case of the Drinfel'd associator and give
the corresponding family between multiple zeta values in equations
\eqref{IMZV}, \eqref{IIMZV} and \eqref{IIIMZV}.

In Section \ref{section:bar} of the article, we explain how  these
families of relations between multiple 
zeta values are induced by iterated integrals on $\m_{0,4}$ and
$\m_{0,5}$ using the bar construction studied by Brown in
\cite{BrownMZVPMS}. The geometry of $\m_{0,5}$ allows us in
Proposition \ref{propitiC5W} to interpret the coefficients $C_{5,W}$
using iterated integrals.
\begin{propa}[{Proposition \ref{propitiC5W}}]
For any bar symbol $\om_W$ dual to a word $W$ in the letters $X_{34}$,
$X_{45}$, $X_{24}$, $X_{12}$, $X_{23}$,
we have
\[
C_{5,W}=\iti_{\gamma} \Reg(\om_{W},\gamma)
\]
where $\Reg(\om, D)$ is the regularization of a bar symbol in
$\oplus\HH^1(\m_{0,5})^{\otimes n}$ along boundary components $D
\subset \partial \m_{0,5}$ and where $\gamma$ is a path around the standard cell
of $\m_{0,5}(\R)$.
\end{propa}
This is a consequence of Theorem \ref{IIIMZVgeom} which links 
the family of relations \eqref{IIIMZV} to the bar construction.
\begin{thma}[{Theorem \ref{IIIMZVgeom}}]
The relation \eqref{III} is equivalent to the family of relations
\[
\forall b_4 \in B_4 \qquad \iti_{\gamma}\Reg(b_{4}^*,\gamma)=0
\]
which is exactly the family of relations \eqref{IIIMZV}. Here 
$(b_4^*)_{b_4 \in  B_4}$ denotes the basis in $V(\m_{0,5})$, the bar
construction on $\m_{0,5}$, dual to the basis $B_4$ of $\UB$ described earlier.
\end{thma}

More generally, we then have that for any basis $F=(f)_{f\in F}$
on $V(\m_{0,5})$, the pentagon relation \eqref{III} is equivalent to 
\[
\forall f \in F \qquad \iti_{\gamma}\Reg(f,\gamma)=0.
\]

Using different methods, and for another purpose, Brown, Gangl and Levin in
\cite{BroGan10} obtain 
the same basis $B_4^*$ of $V(\m_{0,5})$. In their work, the
basis $B_4^*$ is described using combinatorial objects. More precisely, they
use maximal triangulations of rooted decorated polygons.

Instead of looking at all the elements of a basis $F$ of $V(\m_{0,5})$, it
is enough to consider only a subset of $F$ that generates  $V(\m_{0,5})$ as
a shuffle algebra. Indeed, if $\om$ in $V(\m_{0,5})$ is equal to
$f_1\sha f_2$, the iterated integral $\iti_{\gamma} \om$ is equal to
$\iti_{\gamma} f_1 \iti_{\gamma} f_2$. Thus it does not give a new
relation between multiple zeta values. Considering a set of generators
of the shuffle algebra leads to computing many less relations. In degrees
$2$ and $3$ we have respectively $4$ and $10$ generators instead of
$19$ and $65$ elements in the vector space basis. In the appendix, we will give
these relations in degrees $2$ and $3$ using the basis $B_4$. 

The multiplicative generators 
that we have found do not have  a particularly simple expression in terms of
symbols $\om_{W}$ dual to words $W$ in the letters $X_{34}$, $X_{45}$, $X_{24}$,
$X_{12}$, $X_{23}$. But it seems to be linked with our particular choice of
identification. Indeed, using $X_{14}$, $X_{24}$ and $X_{34}$ as generators of
the kernel of $f_4 : \UB \lra \UF$ leads to an other identification:
\[
\UB \simeq  k\llan X_{34},X_{14},X_{24} \rran \rtimes k\llan
X_{12},X_{23}\rran
\] and to another basis $\tilde B_4$ of $\UB$. Then, multiplicative generators 
 can be found with a particularly simple expression in terms of
symbols $\om_{W}$ dual to words $W$ in the letters $X_{34}$, $X_{14}$, $X_{24}$,
$X_{12}$, $X_{23}$. More precisely, writing such a word $W$ as 
\[
W= \sum_{\tilde b_4\in \tilde B_4} l_{\tilde b_4, W} \tilde b_4,
\]
we can write $\tilde b_4^*=\sum_{W} l_{\tilde b_4,W} \om_W$. The multiplicative
generators in low degree are
 elements $\tilde b_4^*$ such that the number of $l_{\tilde b_4,W}$ is as minimal as possible.
This seems to be a general fact.

\section{Combinatorial description of associator relations}
The goal of this section is, for any associator and for the particular
case of $\pkz$, to give an explicit expression for the relations between the 
coefficients derived from the associator relations \eqref{rel1}, \eqref{rel2} 
and \eqref{rel3}. For each of these relations, we will first study the case of 
a general associator and then deduce, for the Drinfel'd associator, relations
between the regularized multiple zeta values.  Let 
\[
\Phi= \sum_{W\in \ax}Z_W W
\]
be an associator. The idea will be to expand the product in the right
hand side of the equations \eqref{rel1}, \eqref{rel2} and
\eqref{rel3} in 
a suitable basis of the space $\UF$ or $\UB$.  Both $\UF$ and $\UB$ can be seen
as a completion of 
  polynomial algebras. Precisely, $\UF$ is the completion of $k\langle X_0,X_1\rangle$, the
  polynomial algebra over $k$ in two non-commutative variables, with respect to
  the ideal generated by $X_0$ and $X_1$. The algebra $\UB$ is the completion
  with respect to the ideal generated by the $X_{ij}$ of the polynomial algebra
  $k\langle X_{ij}\rangle / \mc R$ with $1\leqs i,j \leqs 5$ and where $\mc R$ denotes the
  following relations:
\begin{itemize}
\item $X_{ii}=0$ for $1 \leqs i \leqs 5$,
\item $X_{ij}=X_{ji}$ for $1 \leqs i,j \leqs 5$,
\item $\sum\limits_{j=1}^5 X_{ij}=0$ for $1 \leqs i \leqs 5$, 
\item $[X_{ij}, X_{kl}]=0$ if $\{i,j\}\cap\{k,l\}=\emptyset$.
\end{itemize} 
\begin{defn}\label{defnbasis}
A \emph{basis} $B=(b)_{b \in B}$ of $\UF$ (resp. $\UB$)
  will denote a basis of the underlying vector space of the polynomial algebra
  $k\langle X_0,X_1\rangle$ (resp. $k\langle X_{ij}\rangle/\mc R$) such that 
\begin{itemize}
\item Any element $\Psi$ in $\UF$ (resp. $\UB$) can be uniquely written as a series
\[
\Psi=\sum_{b \in B} a_b b.
\]
\item The elements $b$ in $B$ are homogeneous.
\end{itemize}
\end{defn}
Speaking of a basis of $\UF$ or $\UB$, we will always mean a basis as in the
above definition.
\begin{rem}\label{rembasis} Let $B$ be a basis (as above) of $\UF$
  (resp. $\UB$). Assumptions in Definition \ref{defnbasis} ensure that $1$
  is in $B$ and
\begin{itemize}
\item Any $W$ in $ \ax$ (resp. a word in the letters $X_{ij}$) can be uniquely written
  as
\[
W=\sum_{b \in B} l_{b,W} b \qquad \mx{in }\UF \quad (\mx{resp. in }\UB).
\]
\item Given such a decomposition for $W$, only finitely many $l_{b,W}$ are non
  zero when $b$ runs through $B$.
\item Fixing $b$, only finitely many $l_{b,W}$ are non zero  when $W$ runs
  through $\ax$ (resp. runs through the words in the letters $X_{ij}$).
\end{itemize} 
\end{rem}
\subsection{The symmetry, \eqref{rel1} and \eqref{I}}
Let $P_2$ be
the product 
\[P_2=\Phi(X_0,X_1)\Phi(X_1,X_0).\]
As the monomials in $\UF$,
i.e. the words in $\ax$, form a basis of $\UF$, we can write $P_2$ as 
\[
P_2=\sum_{W\in \ax}C_{2,W} W=1+\sum_{W\in \ax,\, W\neq
  \emptyset}C_{2,W} W. 
\]
The relation \eqref{rel1} tells us that for each $W \in \ax$, $W$ being
nonempty, we have
\begin{equation}\label{rel1Cw}
C_{2,W}=0.
\end{equation}
\begin{exm}In low degree we have the following relations:
\begin{itemize}
\item In degree one, there are just $2$ words: $X_0$ and $X_1$ and
  \eqref{rel1Cw} gives: 
\begin{align*}
C_{2,X_0}=Z_{X_0}+Z_{X_1}=0 \\
C_{2,X_1}=Z_{X_1}+Z_{X_0}=0
\end{align*}
\item In degree two there are $4$ words $X_0X_0$, $X_0X_1$, $X_1X_0$ and
  $X_1X_1$ and \eqref{rel1Cw} gives:
\begin{align*}
C_{2,X_0X_0}=Z_{X_0X_0}+Z_{X_0}Z_{X_1}+Z_{X_1X_1}=0\\
C_{2,X_0X_1}=Z_{X_0X_1}+Z_{X_0}Z_{X_0}+Z_{X_1X_0}=0\\
C_{2,X_1X_0}=Z_{X_1X_0}+Z_{X_1}Z_{X_1}+Z_{X_0X_1}=0\\
C_{2,X_0X_0}=Z_{X_1X_1}+Z_{X_1}Z_{X_0}+Z_{X_0X_0}=0
\end{align*}
\item In degree three there are $8$ words. Looking at the coefficients of the 
words
  $X_0X_0X_1$ in $P_2$, equation \eqref{rel1Cw} gives:
\[
Z_{X_0X_0X_1} +Z_{X_0X_0}Z_{X_0}+Z_{X_0}Z_{X_1X_0}+Z_{X_1X_1X_0}=0
\] 
\end{itemize}
\end{exm}
Let $ \theta $ be the automorphism of $\UF$ that sends $X_0$ to $X_1$ and $X_1$
to $X_0$. Then we have:
\begin{thm}\label{rel1coef} 
The relation \eqref{rel1} is equivalent to the family of relations
\begin{equation}\label{rel1coefeq}
\forall W \in \ax\sm \{\emptyset\}, \qquad \qquad
\sum_{\substack{U_1,U_2 \in \ax\\U_1U_2=W}}Z_{U_1}Z_{\theta(U_2)}=0.  
\end{equation}
\end{thm}
\begin{proof}
As $\Phi(X_1,X_0)= \theta(\Phi(X_0,X_1))$, we have 
\begin{align*}
\Phi(X_1,X_0)&=\theta\left( 1+\sum_{\substack{W \in \ax \\ W\neq \emptyset}}
Z_WW
\right)
= 1+\sum_{\substack{W \in \ax \\ W\neq \emptyset}}
Z_W\theta(W)\\
&
= 1+\sum_{\substack{W \in \ax \\ W\neq \emptyset}}
Z_{\theta(W)}W.\\
\end{align*}
Then, expanding the product $P_2$ and reorganizing, we have 
\begin{align*}
\Phi(X_0,X_1)\Phi(X_1,X_0)&=\left( 
1+\sum_{\substack{U_1 \in \ax \\ U_1\neq \emptyset}}
Z_{U_1}U_1
\right)\left(
1+\sum_{\substack{U_2 \in \ax \\ U_2\neq \emptyset}}
Z_{\theta(U_2)}U_2
\right)\\
&=1+\sum_{\substack{W \in \ax \\ W\neq \emptyset}} \left( 
\sum_{\substack{U_1,U_2 \in \ax\\U_1U_2=W}}Z_{U_1}Z_{\theta(U_2)}
\right)W.
\end{align*}
\end{proof}
\begin{coro}\label{IMZVcomb}
The relation \eqref{I} is equivalent to the family of relations 
\begin{multline}
\forall W \in \ax, \, W\neq \emptyset, \quad \\
\sum_{\substack{U_1,U_2 \in \ax\\U_1U_2=W}} (-1)^{\de(U_1)}\zsha(U_1)
(-1)^{\de(\theta(U_2))}\zsha(\theta(U_2)) =0,\label{IMZV}
\end{multline}
that family being equivalent to the following
\begin{equation}
\forall W \in \ax, \, W\neq \emptyset,  \quad
 \sum_{\substack{U_1,U_2 \in \ax\\U_1U_2=W}}(-1)^{|U_2|} \zsha(U_1)
\zsha(\theta(U_2)) =0.
\end{equation}
\end{coro}
\begin{rem}If $W=X_{\ve_1}\cdots X_{\ve_n}$ is a word in $\ax$, we define
  $\st{\leftarrow}{W}$ to  be the word $W=X_{\ve_n}\cdots X_{\ve_1}$. One can
  then check that the family of relations 
  \eqref{IMZV} (and thus \eqref{I}) is implied by the following:
\begin{enumerate}
\item \emph{Shuffle relations}: 
\[
\text{for all }V\text{  and }W\text{ in }\ax,\qquad \zsha(V \sha
  W)=\zsha(V)\zsha(W).
\] 
\item \emph{Duality relations \cite{Ohnodual, Zagierdual}}:
\[
\text{ for all }W\text{ in }\ax,\qquad 
\zsha(W)=\zsha(\st{\leftarrow}{\theta(W)}).
\] 
\end{enumerate}

The author does not know whether one can deduce the duality relations from the
double shuffle relations. 

The duality relations may be derived from \eqref{I}, that is
\[
\pkz(X_0,X_1)\pkz(X_1,X_0)=1,
\]
and  correspond geometrically to a change of variables $t_i=1-u_i$ in the iterated
integral representation of the multiple zeta values. In order to recover
duality relations directly from \eqref{rel1} and the group-like property, the
argument goes as follows. We want to show that a non-commutative power series in $\UF$
\[
\Phi(X_0,X_1)=1+\sum_{W\in \ax \sm\{
  \emptyset\}}C_W W 
\] 
 which is a group-like element and satisfies the $2$-cycle equation 
\[
\Phi(X_0,X_1)\Phi(X_1,X_0)=1
\]
has coefficients that satisfy the duality relations
\begin{equation}\label{generalduality}
\forall W \in \ax, \quad W\neq \emptyset, \qquad \qquad 
C_{\st{\leftarrow}{\theta(W)}}=(-1)^{\we(W)}C_W.
\end{equation}

Applying this result to the Drinfel'd associator, that is for
\[
C_W=(-1)^{\de(W)}\zsha(W),
\] one derives from \eqref{I} the duality relations
for the multiple zeta values, that is 
\[
 \forall W\, \in \ax\sm\{\emptyset\}, \qquad \qquad 
\zsha(W)=\zsha(\st{\leftarrow}{\theta(W)}). 
\]

To obtain the set of relations \eqref{generalduality}, one should first remark
that  
\begin{align*}
\Phi(X_0,X_1)^{-1}&=1+\sum_{W\in \ax \sm\{\emptyset\}}(-1)^{\we(W)}C_W
\st{\leftarrow}{W}\\
&=1+\sum_{W\in \ax \sm\{\emptyset\}}(-1)^{\we(\st{\leftarrow}{W})}C_{\st{\leftarrow}{W}} W.
\end{align*}
As the group elements are Zariski dense in the group-like elements, one has the above
equality because the inverse of a group element $g=e^{\ve_1X_{i_1}}\cdots
e^{\ve_nX_{i_n}}$, with $X_{i_k}$ in $\{X_0, X_1\}$ and $\ve_i$ in $\{\pm 1\}$, is given by
$g^{-1}=e^{-\ve_nX_{i_n}}\cdots e^{-\ve_1X_{i_1}}$. Then, as 
\[
\Phi(X_1, X_0)=1+\sum_{W\in \ax \sm\{\emptyset\}}C_W\theta(W)
=1+\sum_{W\in \ax \sm\{\emptyset\}}C_{\theta(W)}W,
\]
using the $2$-cycle equation
\eqref{rel1} written as $\Phi(X_1, X_0)=\Phi(X_0, X_1)^{-1} $, one obtains
\[
\forall W \in \ax, \quad W\neq \emptyset, \qquad \qquad 
C_{\theta(W)}=(-1)^{\we(\st{\leftarrow}{W})}C_{\st{\leftarrow}{W}}.
\]
The above set of relations is equivalent to the duality relations \eqref{generalduality}.

\end{rem}
\subsection{The $3$-cycle or the hexagon relation, \eqref{rel2} and
  \eqref{II}} 
For any element $P=\sum_{W\in \ax}a_W W$ in $\UF$, let $\cax(P|W)$ be the
coefficient $a_W$ of the monomial $W$.

Let $P_3$ be the product 
\[P_3= 
e^{\frac{\mu}{2}X_0}\Phi(X_{\infty},X_0)e^{\frac{\mu}{2} X_{\infty}}
\Phi(X_1,X_{\infty})e^{\frac{\mu}{2}X_1}\Phi(X_0,X_1).
\] 
We can write $P_3$ as 
\[
P_3=\sum_{W\in \ax}\cax(P_3|W) W=\sum_{W\in \ax}C_{3,W} W.
\]
The relation \eqref{rel2} tells us that for each $W \in \ax$, $W\neq
\emptyset$,  we have 
\begin{equation}\label{rel2Cw}
C_{3,W}=0.
\end{equation}
In order to make these coefficients explicit, we will need some definitions.
\begin{defn}Let $\alpha_0$ (resp. $\alpha_1$ and $\alpha_{\infty}$) be the
  endomorphism of $\UF$ defined on $X_0$ and $X_1$ by:
\[
\alpha_0(X_0)=X_0 \qquad \mx{and} \qquad \alpha_0(X_1)=0, 
\] 
respectively 
\[
\alpha_1(X_0)=0 \qquad \mx{and} \qquad \alpha_1(X_1)=X_1 
\] 
and 
\[
\alpha_{\infty}= -(\alpha_0+\alpha_1).
\]
Let $\tilde{\alpha}_i$ be the composition of $\alpha_i$ with $X_0,X_1 
\mapsto 1$.  
\end{defn}

The following proposition is a consequence of the expression of the exponential
\[
\forall P\in \UF \qquad \exp(P)= \sum_{n\geqs 0} \frac{P^n}{n!} 
\]
and of the equality 
\begin{equation} \label{eqwordexp}
(-X_0-X_1)^n=\sum_{\substack{W\in \ax\\ |W|=n}}(-1)^{|W|}W.
\end{equation}
\begin{prop}\label{C3exp} Let $W$ be a word in $\ax$. Then
\begin{align*}
\cax(e^{\frac{\mu}{2}X_0}|W)&=
\frac{\mu^{|W|}}{2^{|W|}|W|!}\tilde{\alpha}_0(W),\\
\cax(e^{\frac{\mu}{2}X_1}|W)&=
\frac{\mu^{|W|}}{2^{|W|}|W|!}\tilde{\alpha}_1(W)\qquad \mx{and}\\
\cax(e^{\frac{\mu}{2}X_{\infty}}|W)&=(-1)^{|W|}
\frac{\mu^{|W|}}{2^{|W|}|W|!}.\\
\end{align*}
\end{prop}
In order to describe the coefficient of $\Phi(X_i,X_j)$ with either one of the
variables being $X_{\infty}$, we introduce a set of different decompositions 
of $W$ into sub-words.
\begin{defn}
Let $W$ be a word in $\ax$. For $i\in\{0,1\}$, let 
$
\decx(W,X_i)
$
 be the set of tuples $(V_1, X_i^{k_1}, V_2, X_i^{k_2},\ldots,V_p,X_i^{k_p})$
 with 
\begin{enumerate}
\item $1\leqs  p < \infty$, 
\item $V_j \in \ax$ and  $V_2, \ldots, V_p \neq
 \emptyset$,
\item $k_1, \ldots, k_{p-1} >0$ and $k_p \geqs 0$
\end{enumerate}
 such that  
\[
W=V_1X_i^{k_1} V_2 X_i^{k_2}\cdots V_pX_i^{k_p}.
\]
We will write  
$
(\mb V, \mb k) \in \decx(W,X_i)
$ instead of 
\[(V_1, X_i^{k_1}, V_2, X_i^{k_2},\ldots,V_p,X_i^{k_p})
\in \decx(W,X_i)
\]
 and $|\mb V|$ (resp. $|\mb k|$|) will denote
$|V_1|+\cdots +|V_p|$ (resp. $k_1+\cdots +k_p$). 
\end{defn}
The following proposition describes the coefficient of $W$ in the series
$\Phi(X_0,X_1)$, $\Phi(X_{\infty},X_0)$ and $\Phi(X_1,X_{\infty})$.
\begin{prop}\label{C3Phi}%
Let $W$ be a word in $\ax$. We have 
\[
\cax(\Phi(X_0,X_1)|W)=Z_W.
\]
The coefficients
  $\cax(\Phi(X_{\infty},X_0)|W)$ and $\cax(\Phi(X_1,X_{\infty})|W)$ can be
  written as 
\[
\cax(\Phi(X_{\infty},X_0)|W)=
\sum_{(\mb V, \mb k) \in \decx(W,X_0)}
(-1)^{|\mb V|}
Z_{X_0^{|V_1|} X_1^{k_1}X_0^{| V_2|} X_1^{k_2}\cdots
  X_0^{|V_p|}X_1^{k_p}}
\]
and
\[
\cax(\Phi(X_1,X_{\infty})|W)=
\sum_{(\mb V, \mb k) \in \decx(W,X_1)}
(-1)^{|\mb V|}
Z_{X_1^{|V_1|} X_0^{k_1}X_1^{| V_2|} X_0^{k_2}\cdots X_1^{|V_p|}X_0^{k_p}}.
\]
\end{prop}
\begin{proof} The first statement is immediate. 
Let $\Lcn$ denote the set of double $p$-tuples ( $0\leqs
  p<\infty$) of integers
$  ((l_1,\ldots,l_p),(k_1,\ldots k_p))$ with $k_i,l_i \in \N$, such that, when
$p\geqs 2$ one has $k_i>0 $ for $i=1,\ldots,
{p-1}$, and  $l_j >0$ for $j=2,\ldots, {p}$. Let
$(\mb l, \mb k)$ denote an element of $\Lcn$.
We can write $\Phi(X_{\infty},X_0)$ as 
\[
\Phi(X_{\infty},X_0)=
\sum_{(\mb l, \mb k)\in \Lcn} Z_{X_{0}^{l_1} X_1^{k_1}\cdots  X_{0}^{l_p}X_1^{k_p}} 
X_{\infty}^{l_1} X_0^{k_1}\cdots
X_{\infty}^{l_p}X_0^{k_p} 
\]
which equals 
\[
\sum_{(\mb l, \mb k)\in \Lcn}  Z_{X_{0}^{l_1} X_1^{k_1}\cdots
  X_{0}^{l_p}X_1^{k_p}}
(-1)^{|\mb l|} 
(X_0+X_1)^{l_1} X_0^{k_1}\cdots (X_0+X_1)^{l_p}X_0^{k_p}. 
\]
Reorganizing, we see that the expression of $\cax(\Phi(X_{\infty},X_0)|W)$ follows
from \eqref{eqwordexp}; the case of $\cax(\Phi(X_1,X_{\infty})|W)$ is identical.
\end{proof}
\begin{thm}\label{rel2coef} 
The relation \eqref{rel2} is equivalent to the family of relations
\begin{multline}\label{rel2coefeq} 
\forall W \in \ax\sm\{\emptyset\}, \\
\shoveleft
\sum_{\substack{W_1, \ldots,W_6 \in \ax \\ 
W_1\cdots W_6=W}}
\frac{\mu^{|W_1|}}{2^{|W_1|}|W_1|!} \tilde{\alpha_0}(W_1) \times \\ 
\left(\sum_{\substack{(\mb U, \mb k) \in \\ \decx(W_2,X_0)}} 
(-1)^{|\mb U|} Z_{X_0^{|U_1|} X_1^{k_1}\cdots
  X_0^{|U_p|}X_1^{k_p}} \right) 
(-1)^{|W_3|}\frac{\mu^{|W_3|}}{2^{|W_3|}|W_3|!}  \times \\
\left(\sum_{\substack{(\mb V,\mb l)\in \\ \decx(W_4,X_1)}}
(-1)^{|\mb V|} 
Z_{X_1^{|V_1|} X_0^{l_1}\cdots  X_1^{|V_p|}X_0^{l_p}} \right)
\frac{\mu^{|W_5|}}{2^{|W_5|}|W_5|!} \tilde{\alpha_1}(W_5) 
Z_{W_6} =0.
\end{multline}
\end{thm}
\begin{proof}
The relation \eqref{rel2} is equivalent to the family of relations 
\[
\forall W \in \ax\sm\{\emptyset\} \qquad \qquad \cax(P_3,W)=0.
\]
As $P_3$ is a product of six factors, this is equivalent to 
\begin{multline*}
\forall W \in \ax\sm \{\emptyset\}\\
\shoveleft
\sum_{\substack{W_1, \ldots,W_6 \in \ax \\ 
W_1\cdots W_6=W}}
\cax(e^{\frac{\mu}{2}X_0},W_1)\cax(\Phi(X_{\infty},X_0),W_2)\cdot\\
\cax(e^{\frac{\mu}{2}X_{\infty}},W_3)\cax(\Phi(X_1,X_{\infty}),W_4)\cdot\\
\cax(e^{\frac{\mu}{2}X_{1}},W_5)\cax(\Phi(X_{0},X_1),W_6)=0.
\end{multline*}
The proposition then follows from Proposition \ref{C3exp} and \ref{C3Phi}.
\end{proof}
\begin{coro}\label{IIMZVcom}The relation \eqref{II} is equivalent to the family of
  relations 
 \begin{multline}
 \forall W \in \ax\sm\{\emptyset\}, \\
\shoveleft
\sum_{\substack{W_1, \ldots,W_6 \in \ax \\ 
 W_1\cdots W_6=W}}
 \frac{(i\pi)^{|W_1|}}{|W_1|!} \tilde{\alpha_0}(W_1) \times \\
 \left(\sum_{\substack{(\mb U, \mb k)
     \\ \in \decx(W_2,X_0)}} 
 (-1)^{|W_2|}\zsha(X_0^{|U_1|} X_1^{k_1}\cdots
   X_0^{|U_p|}X_1^{k_p}) \right) 
 (-1)^{|W_3|}\frac{(i\pi)^{|W_3|}}{|W_3|!}  \times \\
 \left(\sum_{\substack{(\mb V, \mb l) \\
 \in \decx(W_4,X_1)}} \hspace{-3ex}
 \zsha(X_1^{|V_1|} X_0^{l_1}\cdots  X_1^{|V_p|}X_0^{l_p}) \right)
\frac{(i\pi)^{|W_5|}}{|W_5|!} \tilde{\alpha_1}(W_5) \times \\
 (-1)^{\de(W_6)}\zsha(W_6) =0. \label{IIMZV}
 \end{multline}
\end{coro}
\subsection{The $5$-cycle or the pentagon relation, \eqref{rel3} and
  \eqref{III}}  
In order to find families of relations between the coefficients
equivalent to \eqref{rel1} and \eqref{rel2}, we decomposed the
product $P_2$ and $P_3$ in the basis of $\UF$ given by the words in $X_0$ and
$X_1$. We will do the same thing here; however, the monomials in the variables
$X_{ij}$ do not form a basis of $\UB$, because there are relations
between the $X_{ij}$. Using the defining relations of $\UB$, we see
that $X_{51}=-X_{12}-X_{13}-X_{14}$, and that
\begin{align*}
X_{51}&= -X_{54}-X_{53}-X_{52}\\
&=2X_{23}+2X_{24}+2X_{34}+X_{12}+X_{13}+X_{14}.
\end{align*}
Then, as the characteristic of $k$ is zero, we have
$X_{51}=X_{23}+X_{24}+X_{34}$. In this section, we will expand the
product in the  R.H.S of \ref{rel3}
using this relation and then decompose this product in a basis
of $\UB$. Let $B$ denote a basis of $\UB$ (in the sense of Definition
\ref{defnbasis}), and let $B_4$ denote the basis of $\UB$ coming from the 
identification  
\[
\UB \simeq k\llan X_{24} , X_{34},X_{45} \rran \rtimes k\llan
X_{12,}X_{23} \rran.
\]
This identification is induced by the morphism $f_4 : \UB \lra \UF$ that
maps $X_{i4}$ to $0$ ($1\leqs i\leqs 5$), $X_{12}$ to $X_0$, $X_{23} $
to $X_1$; the images of the other generators  are easily deduced from
these, by the choice of $X_{24}$, $X_{34}$ and $X_{45}$ as generators of the
kernel of $f_4$ (see \cite{IharaAPSBG}). Using the relation defining $\UB$, one sees that
\[
[X_{ij},X_{jk}]=-[X_{ik},X_{jk}] \qquad i\neq j,k \mx{ and } j\neq k
\]
which gives for example
\[
[X_{12},X_{24}]=-[X_{14},X_{24}]=[X_{34},X_{24}]+[X_{45},X_{24}].
\]

The basis $B_4$ is formed by $1$ and  the monomials, that is words of the form $U_{245}V_{123}$
where $U_{245}$ is a word in $\axq=\{X_{24},X_{34},X_{45}\}^*$ and $V_{123}$ is in
$\axud=\{X_{12},X_{23}\}^*$. Speaking of the empty word $\emptyset$ in $B_4$, we will
mean $1$ when seen in $\UB$ and $\emptyset$ when seen as the word.

Let $\axb$ be the dictionary $\{X_{24},X_{34},X_{45},X_{12},X_{23}\}^*$, and
let $\axtq$ and $\axudtq$ be respectively the sub-dictionary

\[
\axtq=\{X_{23},X_{24},X_{34}\}^*\qquad \mx{and}
\qquad \axudtq=\{X_{12},X_{23},X_{24},X_{34}\}^*.
\]  

Let $P_5$ be the product in $\UB$.
\[
\Phi(X_{12},X_{23}) \Phi(X_{34},X_{45}) \Phi(X_{51},X_{12})
\Phi(X_{23},X_{34}) \Phi(X_{45},X_{51}).
\]

As $X_{51}= X_{23}+X_{24}+X_{34}$, we can write $P_5$ without using
$X_{51}$
\begin{multline*}
P_5=\Phi(X_{12},X_{23}) \Phi(X_{34},X_{45}) \Phi( X_{23}+X_{24}+X_{34},X_{12})
\Phi(X_{23},X_{34}) \\
\Phi(X_{45}, X_{23}+X_{24}+X_{34}).
\end{multline*}
Expanding the terms $(X_{23}+X_{24}+X_{34})^n$ as 
\[
\sum_{\substack{W \in \axtq\\ |W|=n}} W,
\]
we have 
\begin{equation}\label{C5Wdef}
P_5=\sum_{W \in \axb} C_{5,W} W.
\end{equation}
Despite the fact that this expression is not unique as a decomposition of $P_5$
in $\axb$, these $C_{5,W}$ are the coefficients of a word $W$ just after expanding the product 
$P_5$ without $X_{51}$ (that is replacing $X_{51}$ by $X_{23}+X_{24}+X_{34}$), and as
such, they are unique and well defined. 

\begin{defn}\label{pgdef} Let $\pa$, $\pb$, $\pc$, $\pd$, $\pe$ be the morphisms
  from $\UB$ to $\UF$ defined respectively  on the monomial $X_{12}$, $X_{23}$,
  $X_{34}$, $X_{45}$, $X_{24}$ by: 
\newcommand{\espa}{\quad}
\[\hspace{-4pt}
\begin{array}{l@{\hspace{0.32em}}l@{\hspace{0.32em}}l@{\hspace{0.32em}}l%
@{\hspace{0.32em}}l}
\pa(X_{12}) =X_0, &\pa(X_{23})=X_1, &\pa(X_{34})=0, &
\pa(X_{45})=0, &\pa(X_{24})=0, \\
\pb(X_{12}) =0, &\pb(X_{23})=0, &\pb(X_{34})=X_0, &
\pb(X_{45})=X_1, &\pb(X_{24})=0, \\
\pc(X_{12}) =X_1, &\pc(X_{23})=X_0, &\pc(X_{34})=X_0, &
\pc(X_{45})=0, &\pc(X_{24})=X_0, \\
\pd(X_{12}) =0, &\pd(X_{23})=X_0, &\pd(X_{34})=X_1, &
\pd(X_{45})=0, &\pd(X_{24})=0, \\
\pe(X_{12}) =0, &\pe(X_{23})=X_1, &\pe(X_{34})=X_1, &
\pe(X_{45})=X_0, &\pe(X_{24})=X_1. 
\end{array}
\]
By convention, we will have $\pg(1)=\pg(\emptyset)=1$.
\end{defn}
\begin{prop}\label{C5W}
For all words $W \in \axb$ ($W\neq \emptyset$), the coefficient $C_{5,W}$ is given by
\begin{equation}
C_{5,W}=\sum_{\substack{U_1,\ldots,U_5 \in \axb \\ U_1\cdots U_5=W}}
Z_{\pa(U_1)}Z_{\pb(U_2)}Z_{\pc(U_3)}Z_{\pd(U_4)}Z_{\pe(U_5)},
\end{equation}
where by convention $Z_{0}=0$ and $Z_{1}=Z_{\emptyset}=1$.
\end{prop} 
\begin{proof} It is enough to show that the $i$-th factor of $P_5$
  without using $X_{51}$ can be written as 
\[
\sum_{U_i \in \axb} Z_{\pg(U_i)}U_i. 
\]
As the first, second and fourth factors are similar, we will discuss only the
first one.
It is clear in the case of $\Phi(X_{12},X_{23})$ that either
$U_1$ is in $\axud$ and its coefficient is then $Z_{\pa(U_1)}$,
or $U_1$ is not in $\axud$ and it
does not appear in $\Phi(X_{12},X_{23})$ which means that its coefficient is
$0$.  

The third and fifth factors are similar and thus we will treat only the former. 
We can write $\Phi( X_{23}+X_{24}+X_{34},X_{12})$ as 
\[
\sum_{(\mb l, \mb k )\in \Lcn} 
Z_{X_0^{l_1}X_1^{k_1}\cdots X_0^{l_p}X_1^{k_p}} 
(X_{23}+X_{24}+X_{34})^{l_1}X_{12}^{k_1} 
\cdots (X_{23}+X_{24}+X_{34})^{l_p}X_{12}^{k_p}.
\]
We can rewrite  the previous sum as running over all the words in the
letters $X_{12}$, $X_{23}$, $X_{24}$ and $X_{34}$ because $(X_{23}+X_{24}+X_{34})^{l}$ is equal to
\[
\sum_{\substack{W\in \axtq \\ |W|=l}} W.
\]
Using the unique decomposition (as word) of $U_3\in \axudtq$ as
\[
U_3=V_1X_{12}^{k_1} \cdots V_pX_{12}^{k_p}\qquad \mx{with }V_i \in \axtq,
\] 
we see that each word $U_3$ in $\axudtq$ appears one and only one time in $\Phi(
X_{23}+X_{24}+X_{34},X_{12})$ with the
coefficient $Z_{X_0^{|V_1|}X_1^{k_1}\cdots X_0^{|V_p|}X_1}$. We finally have 
\[
\Phi( X_{23}+X_{24}+X_{34},X_{12})=\sum_{U_3 \in \axb} Z_{\pc(U_3)} U_3.
\]
\end{proof}

We fix a basis $B$ of $\UB$ (in the sense of Definition \ref{defnbasis}). Remark
\ref{rembasis} ensures that
for every $W$ in $\axb$, there exists a unique decomposition of $W$ (in  $\UB$)
in terms of linear 
combinations of elements of $B$
\[
W=\sum_{b\in B} l_{b,W} b \qquad l_{b,W} \in k.
\]
Then, using the basis $B$, we can find a family of relations equivalent to \eqref{rel3}.
\begin{thm}\label{rel3coef}The relation \eqref{rel3} is equivalent to the family of relations
\begin{flalign}\label{rel3coefeq}
\forall b \in B\, \, (b\neq 1) & \qquad \qquad 
\sum_{W\in \axb} l_{b,W} C_{5,W} =0 & &
\end{flalign}
where $C_{5,W}$ are given by Proposition \ref{C5W}.
\end{thm}
\begin{proof}
As observed in Remark \ref{rembasis}, for a given $W$ in $\axb$ there are only finitely
many $l_{b,W}$ 
that are non zero. Moreover, for any $b$ in $B$ there are only finitely many
$l_{b,W}$ that are non zero.

The product $P_5$ is then equal to 
\begin{align*}
P_5&=\sum_{W \in \axb} C_{5,W} W\\
&=\sum_{W \in \axb} C_{5,W}\left(\sum_{b\in B} l_{b,W} b \right)\\
&=\sum_{b\in B}\left(\sum_{W \in \axb} l_{b,W}C_{5,W} \right)b.
\end{align*}
The relation \eqref{rel3} tells us that 
\[
P_5=1
\]
which, because $1$ is in $B$, means that $C_{5,\emptyset}=1$ and 
\begin{flalign*}
\forall b \in B\, \, (b\neq 1) & \qquad \qquad 
\sum_{W\in \axb} l_{b,W} C_{5,W} =0. & &
\end{flalign*}
\end{proof}
Using the more common basis $B_4$ we have:
\begin{coro}
The relation \eqref{rel3} is equivalent to the family of relations
\begin{flalign}\label{rel3coefeqB4}
\forall b_4 \in B_4\, \, (b_4\neq 1) & \qquad \qquad 
\sum_{W\in \axb} l_{b_4,W} C_{5,W} =0 & &
\end{flalign}
where the $C_{5,W}$ are given by Proposition \ref{C5W}.
\end{coro}
\begin{rem}
In the case of the basis $B_4$ one can check that the coefficients
$l_{b_4,W}$ are in $\Z$.
\end{rem}
The previous corollary, applied to the particular case of the Drinfel'd
associator and making explicit the $C_{5,W}$ in terms of multiple zeta values, gives:
\begin{thm}\label{IIIMZVcom} With the convention that $\zsha(0)=0$,
the relation \eqref{III} is equivalent to the family of relations
\newlength{\hauteurpar}
\begin{multline}\label{IIIMZV}
\forall b_4 \in B_4\, \, (b_4\neq 1) \\
\shoveleft 
\sum_{W
}l_{b_4,W} 
\left(
\sum_{
U_1\cdots U_5=W}
(-1)^{\dea(U_1)+\deb(U_2) +\dec(U_3)+ \ded(U_4)+ \dee(U_5)} \right.\\
\left. 
\vphantom{%
\sum_{
U_1\cdots U_5=W}
}
\zsha(\pa(U_1))
\zsha(\pb(U_2))
\zsha(\pc(U_3)
\zsha(\pd(U_4))
\zsha(\pe(U_5))
\rule{0pt}{\hauteurpar} 
\right)
 =0 
\end{multline}
where $\degg(U)$ is the depth of $\pg(U)$ and the words $W$, $U_i$ are in $\axb$.
\end{thm}
\section{Bar Construction and associator relations}\label{section:bar}
In this section, we suppose that $k$ is $\C$. We review the notion of bar construction and its
links with multiple zeta values. Those results have been shown in
greater generality in \cite{IIDFLSChen} and \cite{BrownMZVPMS}. We
will recall Brown's variant of Chen's reduced bar construction in
the case of the moduli spaces of curves of genus $0$ with $4$ and $5$
marked points, $\m_{0,4}$ and $\m_{0,5}$.
\subsection{Bar Construction}
The moduli space of curves of genus $0$ with $4$ marked points, $\m_{0,4}$, is 
\[
\m_{0,4}=\{(z_1,\ldots, z_4) \in (\p^1)^4\, |\, z_i \neq z_j \mx{ if }
i\neq j\}/\PGL_2(k)
\]
 and is identified as 
\[
\m_{0,4}\simeq\{t\in (\p^1)\, |\, t \neq0,1, \infty\}
\]
by sending the point $[(0,t,1,\infty)]\in \m_{0,4}$ to $t$.
 
The moduli space of curves of genus $0$ with $5$ marked points, $\m_{0,5}$, is 
\[
\m_{0,5}=\{(z_1,\ldots, z_5) \in (\p^1)^5\, |\, z_i \neq z_j \mx{ if }
i\neq j\}/\PGL_2(k)
\]
and is identified as 
\[
\m_{0,5}\simeq\{(x,y)\in (\p^1)^2\, |\, x,y \neq0,1, \infty \mx{ and
}x\neq y\}
\]
by sending the point $[(0,xy,y,1,\infty)]\in \m_{0,5}$ to $(x,y)$. This
identification can be interpreted as the composition of 
\[
\xymatrix@R=1ex{
\m_{0,5} \ar[r] & \m_{0,4} \times  \m_{0,4} \\
[(z_1,\ldots,z_5)] \ar@{|->}[r] & [(z_1,z_2,z_3,z_5)]\times [(z_1,z_3,z_4,z_5)]  
}
\]
with the previous identification of $\m_{0,4}$ using the fact that
\[
[(0,xy,y,\infty)]=[(0,x,1,\infty)].
\]

For $\mc M=\m_{0,4}$ or $\mc M=\m_{0,5}$, Brown has defined in
\cite{BrownMZVPMS} a graded Hopf $k$-algebra
\begin{equation}
V(\mc M)= \oplus_{m=0}^{\infty} V_m(\mc M) \subset
\oplus_{m=0}^{\infty} \Hdr^1(\mc M)^{\otimes m}.
\end{equation}
Here $V_0(\mc M)=k$, $V_1(\mc M)=\Hdr^1(\mc M)$ and $V_m(\mc M)$ is the
intersection of the kernel $\wedge_i$ for $1\leqs i \leqs m-1$:
\[
\xymatrix@R=1ex{
\wedge_i :\Hdr^1(\mc M)^{\otimes m} \ar[r] & \Hdr^1(\mc M)^{\otimes
  m-i-1} \otimes \Hdr^2(\mc M) \otimes \Hdr^1(\mc M)^{\otimes i-1} \\
\nu_m \otimes \cdots \otimes \nu_1\ar@{|->}[r] & \nu_m \otimes
\cdots \otimes (\nu_{i+1} \wedge \nu_{i})\otimes \cdots \otimes \nu_1 .
}
\]
Suppose that $\om_1, \ldots, \om_k$ form a basis of $\Hdr^1(\mc M)$;
then the elements of $V_m(\mc M)$ can be written as linear combinations of symbols
\[
\sum_{I=(i_1, \ldots , i_m)} c_I[\om_{i_m}|\ldots|\om_{i_1}],
\]
with $c_I \in k$, which satisfy the integrability condition
\begin{equation}\label{intcond}
\sum_{I=(i_1, \ldots , i_m)} c_I \om_{i_m} \otimes
\cdots \otimes \om_{i_{j+2}}\otimes (\om_{i_{j+1}} \wedge
\om_{i_{j}})\otimes \om_{i_{j-1}} \otimes\cdots \otimes \om_{i_1} =0
\end{equation}
for all $1\leqs j \leqs m-1$.
\begin{defn} Brown's bar construction over $\mc M$ is the tensor product 
\[
B(\mc M)=\mc O_{\mc M} \otimes V(\mc M).
\]
\end{defn}
\begin{thm}[{\cite{BrownMZVPMS}}]
The bar construction $B(\mc M)$ is a commutative graded Hopf algebra isomorphic
to the $0^{Th}$ cohomology group of Chen's reduced bar complex on
$\mc O_{\mc M}$: 
\[
B(\mc M) \simeq \HH^0(B(\Omega^{\bullet}\mc O_{\mc M})).
\]
\end{thm}

Let $\nu_m, \ldots, \nu_1$ be $m$ holomorphic $1$-forms in $\Omega^1(\mc M)$. The
iterated integral of the word $\nu_m \cdots \nu_1$, denoted by 
\[
\iti \nu_m \circ \cdots \circ \nu_1,
\]
is the application that sends any path $\gamma : [0,1] \ra \mc M$
to 
\[
\iti_{\gamma} \nu_m \circ \cdots \circ \nu_1= \int_{0<t_1<\ldots <t_m}
\gamma^*\nu_1(t_1) \w \cdots \w \gamma^*\nu_m(t_m). 
\]
This value is called the iterated integral  of $\nu_m \cdots \nu_1$
along  $\gamma$. We extend these definitions by linearity to linear
combinations of forms $\sum_Ic_I \nu_{i_m} \ldots \nu_{i_1}$. 

When, for any $\gamma$, the iterated integral 
\[
\iti_{\gamma} \sum_I c_I \nu_m \circ \cdots \circ \nu_1
\]
 depends only on the homotopy class of $\gamma$,  
we say it is an homotopy invariant iterated integral and denote it by $\iti  \sum c_I
\nu_m \circ \cdots \circ \nu_1$. Let $L(\mc M)$ denote the set of all  
homotopy invariant iterated integrals.

\begin{prop}[{\cite{BrownMZVPMS}}]
The morphism $\rho$ defined by 
\[
\xymatrix@R=1ex{
\rho :B(\mc M)\ar[r] & L(\mc M)\\
\ds \sum_I c_I [\om_{i_m}|\cdots|\om_{i_1}]\ar@{|->}[r] &
\ds 
\iti \sum_{I} c_I \om_{i_m}
\circ \cdots \circ \om_{i_1} 
}
\]
is an isomorphism.
\end{prop}
\begin{rem}In particular for any such $\gamma$ homotopically equivalent to zero,
  we have for all $\sum_I c_I [\om_{i_m}|\cdots|\om_{i_1}] $ in $ V(\mc M) $:
\[\sum_{I} c_I\iti_{\gamma} \om_{i_m}
\circ \cdots \circ \om_{i_1} = 0
\]
\end{rem}
\subsection{Bar Construction on $\m_{0,4}$, symmetry and hexagon
  relations}
Here, we will show how the symmetry relations \eqref{I} and the hexagon \eqref{II}
relations are related to the bar construction on $\m_{0,4}$.

First of all we should remark that $B(\m_{0,4})$ is extremely
simple. 
\begin{prop}Let $\om_0$ and $\om_1$ denote respectively the
  differential $1$-form, in $\Om^1(\m_{0,4})$, $\frac{\dt}{t}$ and
  $\frac{\dt}{t-1}$.

Then, any element $[\om_{\ve_n}|\cdots|\om_{\ve_1}]$ with $\ve_i$ in
$\{0,1\}$ is an element of $V(\m_{0,4})$. Moreover, the family of
these elements is a basis of $V(\m_{0,4})$.
\end{prop}
\begin{proof}
As $\om_0 \w \om_1=0$, the integrability condition \eqref{intcond} is
automatically 
satisfied, so any element $[\om_{\ve_n}|\cdots|\om_{\ve_1}]$ ($\ve_i=0,1$)
is an element of $V(\m_{0,4})$. Moreover, as $(\om_0,\om_1)$ is a basis of
$\Hdr^1(\m_{0,4})$, the elements $[\om_{\ve_n}|\cdots|\om_{\ve_1}]$ form a basis
of $V(\m_{0,4})$.
\end{proof}

Sending $X_0$ to $\om_0$ and $X_1$ to $\om_1$ gives a one to one correspondence 
between words $W=X_{\ve_n} \cdots X_{\ve_1}$ in $\ax$ and the elements
$[\om_{\ve_n}|\cdots|\om_{\ve_1}]$ of the
previous basis of $V(\m_{0,4})$. This correspondence allows us to identify
$V(\m_{0,4})$ with the graded dual of $\UF$, 
\[
V(\m_{0,4})\simeq (\UF)^{*}.
\]   
The word $W=X_{\ve_n} \cdots X_{\ve_1}$ is sent to its dual 
$W^*=\om_{W}=[\om_{\ve_n}|\cdots|\om_{\ve_1}]$.

\begin{rem}Let $\alpha$ and $\beta$ be two paths in a variety with
$\alpha(1)=\beta(0)$. We will denote by  $\beta  \circ \alpha$ the composed path
beginning with $\alpha$ and ending with $\beta$. 

The iterated integral of
$\om=\om_n \cdots \om_1$ along $\beta \circ \alpha$ is then equal to 
\begin{equation}\label{pathcompo}
\sum_{k=0}^n \left(
\iti_{\beta} \om_n \circ \cdots \circ \om_{n-k+1}
\right)
\left(
\iti_{\alpha} \om_{n-k} \circ \cdots \circ \om_1
\right).
\end{equation}
\end{rem}

Following \cite{BrownMZVPMS} and considering the three dihedral structures on
$\m_{0,4}$, one can define $6$ tangential base points: $\vec{01}$,
$\vec{10}$, $\vec{1\infty}$, $\vec{\infty 1}$,$\vec{\infty 0}$ and
$\vec{0\infty}$. Let $p$ denote the path beginning at the tangential base
point $\vec{01}$ and ending at $\vec{10}$ defined by $t \mapsto t$ and
let $p^{-1}$ denote its 
inverse $t \mapsto 1-t$.

If $\gamma$ is a path, starting at  a tangential
base point $\vec{P}$ (and/or ending at a tangential base point
$\vec{P'}$) an iterated integral $\iti \om$ may be 
divergent. However, one can give (as in \cite{BrownMZVPMS}) a value
to that divergent integral; we speak of the regularized iterated integral. 

If $W$ is a word in $\xay$, the iterated integral $\iti_p \om_W$ is
convergent and is equal to $(-1)^{\de(W)}\zeta(W)$. If $W$ is a word beginning 
by $X_1$ and/or ending by $X_0$ (that is in $\ax
\sm \xay$), then the regularized iterated integral $\iti_p \om_W$ is
equal to $(-1)^{\de(W)}\zsha(W)$. 

We may, thereafter, omit the term
\emph{regularized} in the expressions ``regularized iterated
integral'' or ``regularized homotopy invariant iterated integral''.

\begin{thm}\label{IMZVgeom}
The relation \eqref{I} is equivalent to the family of relations
\[
\forall W\in \ax \qquad 
\iti_{p\circ p^{-1}} \om_{W} =0,
\]
which is exactly the family \eqref{IMZV}.
\end{thm}
\begin{proof}
Considering the KZ equation \eqref{KZ}
\[
\frac{\partial g}{\partial u} = 
\left(\frac{X_0}{u}+\frac{X_1}{u-1}
\right)\cdot g(u)
\]
and the two normalized solutions at $0$ and $1$, $g_0$ and $g_1$,
$\pkz(X_0,X_1)$ is the unique element in $\UF$ such that 
\[
g_0(u)=g_1(u)\pkz(X_0,X_1).
\] 
Using the symmetry of the situation we also have
\[
g_1(u)=g_0(u)\pkz(X_1,X_0).
\]
The equation \eqref{I} comes from the uniqueness of such a solution
normalized at $1$:
\begin{equation}\label{normsolat1}
g_1(u)=g_1(u)\pkz(X_0,X_1)\pkz(X_1,X_0).
\end{equation}
The elements $\pkz(X_0,X_1)$ and $\pkz(X_1,X_0)$ can be expressed
using regularized iterated integrals as 
\[
\pkz(X_0,X_1)=\sum_{W \in \ax} \left(\iti_{p} \om_W \right)W
\]
and 
\[
\pkz(X_1,X_0)=\sum_{W \in \ax} \left(\iti_{p^{-1}} \om_W \right)W.
\]
Equation \eqref{normsolat1}  corresponds to the comparison of the normalized
solution $g_1$ with the solution given by analytic continuation of $g_1$ along
$p\circ p^{-1}$. The product 
\[
\pkz(X_0,X_1)\pkz(X_1,X_0)
\]
is then the series
\[
\sum_{W \in \ax} \left(\iti_{p \circ p^{-1}} \om_W \right)W.
\]
As the path $p \circ p^{-1}$ is homotopically equivalent to $0$, all the previous
  iterated integrals (for $W \neq \emptyset$) are $0$. We deduce that \eqref{I} is
  equivalent to  
\[
\forall W\in \ax \qquad 
\iti_{p\circ p^{-1}} \om_{W} =0.
\]

Now, fix any $W=X_{\ve_n} \cdots X_{\ve_1}$ in $\ax$ and compute the regularized iterated
integral $\iti_{p\circ p^{-1}} \om_{W} $. Using \eqref{pathcompo}, we have 
\[
\iti_{p\circ p^{-1}} \om_{W} = \sum_{k=0}^{n}\left( \iti_p \om_{\ve_n} \circ
\cdots \circ \om_{\ve_{n-k+1}}\right)\left( \iti_{p^{-1}} \om_{\ve_{n-k}} \circ
\cdots \circ \om_{\ve_1}\right).
\]
Setting $U_1=X_{\ve_n} 
\cdots  X_{\ve_{n-k+1}}$ and $U_2=X_{\ve_{n-k}} 
\cdots X_{\ve_1}$, we have 
\[
\iti_p \om_{\ve_n} \circ
\cdots \circ \om_{\ve_{n-k+1}} = (-1)^{\de(U_1)} \zsha(U_1).
\]
As $p^{-1}$ is given by $t \mapsto 1-t$, we have for $\ve$ in $\{0,1\}$
\[
(p^{-1})^*(\om_{\ve}) =\om_{1-\ve}.
\] 
Moreover, as $p^*(\om_{\ve})=\om_{\ve}$, one computes
\begin{align*}
 \iti_{p^{-1}} \om_{\ve_{n-k}} \circ
\cdots \circ \om_{\ve_1} &= 
\int_{0<t_1 < \ldots <t_{n-k}}
(p^{-1})^*(\om_{\ve_1} (t_1))\w \cdots \w
(p^{-1})^*(\om_{\ve_{n-k}}(t_{n-k})) \\
&= 
\int_{0<t_1 < \ldots <t_{n-k}}
\om_{1-\ve_1} (t_1)\w \cdots \w
\om_{1-\ve_{n-k}}(t_{n-k}) \\
&=
 \iti_{p} \om_{1-\ve_{n-k}} \circ
\cdots \circ \om_{1-\ve_1} =\iti_p \om_{\theta (U_2)},
\end{align*}
where $\theta$ exchanges $X_0$ and $X_1$.
Finally, we obtain
\[
\iti_{p^{-1}} \om_{U_2}=\iti_p \om_{\theta(U_2)}=
(-1)^{\de(\theta(U_2))}\zsha(\theta(U_2)) 
\]
and 
\[
0=\iti_{p\circ p^{-1}} \om_{W} = \sum_{U_1U_2=W} (-1)^{\de(U_1)}
\zsha(U_1)(-1)^{\de(\theta(U_2))}\zsha(\theta(U_2)) 
\]
which is exactly the relation \eqref{IMZV} for the word $W$.
\end{proof}

Now, let $c$ be the infinitesimal half circle around $0$ in the lower half
plane, connecting the tangential base point $\vec{0\infty}$ and
$\vec{01}$. The path $c$ can be seen as the limit when $\ve $ tends to $0$ of
$c_{\ve} : t \mapsto \ve e^{i(\pi+t\pi)}$.

 We have a natural $3$-cycle on
$\m_{0,4}$ given by $\tau : t \mapsto \frac{1}{1-t}$. Let $\gamma$ be the
path $ c \circ \tau^2(p)\circ\tau^2(c)\circ \tau (p) \circ \tau(c)\circ p$.

\begin{thm}\label{IIMZVgeom}
The relation \eqref{II} is equivalent to the family of relations
\[
\forall W\in \ax \qquad 
\iti_{\gamma} \om_{W} =0
\]
which is exactly the family \eqref{IIMZV}.
\end{thm}
\begin{proof}
Comparing the six different normalized solutions of \eqref{KZ} at the six
different base points leads to six equations. Combining these equations, one
obtains \eqref{II} via the relation
\begin{equation}\label{normsolat0II}
g_0(u)= g_0(u)e^{i\pi X_0}\pkz(X_{\infty},X_0)e^{i\pi X_{\infty}}
\pkz(X_1,X_{\infty})e^{i\pi X_1}\pkz(X_0,X_1)
\end{equation}
where exponentials are coming from the relation between the solutions at the based
points $\vec{0\infty}$ and $\vec{01}$ (resp. $\vec{10}$ and $\vec{1\infty}$,
$\vec{\infty1}$ and $\vec{\infty0}$), that is, from the monodromy around $0$, $1$ and
$\infty$. 

Putting the six different relations together in order to get the previous equation  
is the same as comparing the solution $g_0$ with the analytic continuation of $g_0$
along any path starting at $\vec{01}$, joining the
other tangential base points  $\vec{10}$, $\vec{1\infty}$, $\vec{\infty1}$,
$\vec{\infty0}$, $\vec{0\infty}$ in that order and ending at $\vec{01}$ ;
staying all the time in the lower half plan. Such a path is homotopically
equivalent to $\gamma$.

Thus, equation \eqref{normsolat0II} gives a relation between $g_0$ and the
solution obtained
from $g_0$ by analytic continuation along $\gamma$. Then, the product in $\UF$ in the
R.H.S of \eqref{normsolat0II} can be 
expressed using homotopy invariant iterated integrals as
\begin{multline*}
e^{i\pi X_0}\pkz(X_{\infty},X_0)e^{i\pi X_{\infty}}
\pkz(X_1,X_{\infty})\\e^{i\pi X_1}\pkz(X_0,X_1) = \sum_{W \in \ax} \left(
  \iti_{\gamma} \om_W \right)W .
\end{multline*}

As $\gamma$ is homotopically equivalent to $0$, for any word $W$ in $\ax$, one has
\[
\iti_{\gamma} \om_{W}=0.
\]
This proves the first part of the theorem. 

Using the decomposition of iterated integrals on a composed path (Equation
\eqref{pathcompo}), we have 
\begin{multline*}
\forall W\in \ax \qquad 
\iti_{\gamma} \om_{W} =
\sum_{\substack{U_1, \ldots, U_6\\ U_1 \cdots U_6=W}} 
\iti_{c} \om_{U_1} \iti_{\tau^2(p)} \om_{U_2} \iti_{\tau^2(c)} \om_{U_3}
\\
\iti_{\tau(p)} \om_{U_4} \iti_{\tau(c)} \om_{U_5} \iti_{p} \om_{U_6} . 
\end{multline*}

Thus, in order to show that the family of relations 
\[
\forall W \in \ax \qquad \iti_{\gamma}\om_W=0
\]
gives exactly the family of relations \eqref{IIMZV}, it
is enough to show that  for any $U$ in $\ax$,
\begin{align*}
&\iti_{c} \om_{U}=\cax(e^{i\pi X_0}|U),\qquad \iti_{\tau^2(p)}
\om_{U}=\cax(\pkz(X_{\infty},X_0)|U),\\ 
&\iti_{\tau^2(c)}\om_{U}=\cax(e^{i\pi X_{\infty}}|U),\qquad
\iti_{\tau(p)}\om_{U} =\cax(\pkz(X_1,X_{\infty})|U),\\ 
&\iti_{\tau(c)} \om_{U}=\cax(e^{i\pi X_1}|U), \qquad 
\iti_{p} \om_{U}=\cax(\pkz(X_0,X_1)|U).
\end{align*}
In order to compute the iterated integral along $c$, $\tau(c)$ and $\tau^2(c)$,
it is enough to compute the limit when $\ve$ tends to $0$ of the iterated
integral along $c_{\ve}$, $\tau(c_{\ve})$ and $\tau^2(c_{\ve})$. As
\[
c_{\ve}^*(\om_0)=i\pi \dt\qquad \mx{and} \qquad c_{\ve}^*(\om_1)= 
\ve \frac{-i\pi  e^{i(\pi+\pi t)}\dt}{1-\ve  e^{i(\pi+\pi t)}},
\]
 the iterated integral $
\iti_{c_{\ve}}\om_U$ tends to $0$ except if $U=X_0^n$ and then
$\iti_{c_{\ve}}\om_{X_0^n}=\frac{(i\pi)^n}{n!}$ for all $\ve$. Thus, we have 
\[
\iti_{c} \om_{U}=\cax(e^{i\pi X_0}|U).
\]

Similarly we have 
\[
\tau(c_{\ve})^*(\om_1)=i\pi \frac{\dt}{1-\ve  e^{i(\pi+\pi t)}} \qquad 
\mx{and} \qquad 
\tau(c_{\ve})^*(\om_0)= 
\frac{\ve  i\pi e^{i(\pi+\pi t)}\dt}{1-\ve  e^{i(\pi+\pi t)}}.
\]
The iterated integral $
\iti_{\tau(c_{\ve})}\om_U$ tends to $0$ unless $U=X_1^n$, and then
$\iti_{c_{\ve}}\om_{X_0^n}$ tends to $\frac{(i\pi)^n}{n!}$ when $\ve$ tends to
$0$. Thus, we have 
\[
\iti_{\tau(c)} \om_{U}=\cax(e^{i\pi X_1}|U).
\]
Computing $\tau^2(c_{\ve})^*$, we have
\[
\tau^2(c_{\ve})^*(\om_0)=-i\pi \frac{\dt}{1-\ve  e^{i(\pi+\pi t)}} \qquad 
\mx{and} \qquad 
\tau^2(c_{\ve})^*(\om_1)= -i\pi \dt .
\]
Then, we find that the limit when $\ve$ tends to $0$ of $
\iti_{\tau^2(c_{\ve})}\om_U$ is $\frac{(-i\pi)^{|U|}}{|U|!}$, which gives
\[
\iti_{\tau^2(c)}\om_{U}=\cax(e^{i\pi X_{\infty}}|U).
\]

The equality 
\[
\iti_{p} \om_{U}=\cax(\pkz(X_0,X_1)|U)
\] 
is obvious. 

Cases of 
\[
\iti_{\tau(p)} \om_{U}\qquad \text{and} \qquad \iti_{\tau^2(p)} \om_{U}
\]
 are extremely similar and we will discuss only the last one. First, we should remark that
\[
(\tau^2)^*(\om_0)= -\om_0+\om_1 \qquad
\mx{and} \qquad
(\tau^2)^*(\om_1)=-\om_0.
\]

For $U=X_{\ve_n} \cdots X_{\ve_1}$ ($\ve_i =0,1$) we can rewrite the
iterated integral 
$\iti_{\tau^2(p)} \om_U$ as 
\[
\iti_p (\tau^2)^*(\om_{\ve_n}) \circ \cdots \circ (\tau^2)^*(\om_{\ve_1}).
\]
We will now prove by induction on $n=|U|$ that in $V(\m_{0,4})$
\begin{multline}\label{IIitip}
[ (\tau^2)^*(\om_{\ve_n})
| \cdots | (\tau^2)^*(\om_{\ve_1})]=\\
\sum_{(\mb V, \mb k) \in \decx(U,X_0)} 
(-1)^{|\mb V|}
\om_{X_0^{|V_1|} X_1^{k_1}X_0^{| V_2|} X_1^{k_2}\cdots
  X_0^{|V_p|}X_1^{k_p}} 
\end{multline}
which will give using Proposition \ref{C3Phi} the equality 
\[
\iti_{\tau^2(p)} \om_{U}=\cax(\pkz(X_0,X_1)|U).
\]

If $U=X_0$, the set $\decx(U,X_0)$ has $2$ elements $((X_0),(0))$ and
$((\emptyset),(1))$. Similarly, if $U=X_1$ then $\decx(U,X_0)$ has only one
element which is $((X_1),(0))$. In both cases \eqref{IIitip} is satisfied.

Let $U=X_{\ve_n} \cdots X_{\ve_1}$ be a word in $\ax$, and
let $\ve$ be in $\{0,1\}$. For the simplicity of notation, we
shall write $[\om_{\ve_n}|\cdots
|\om_{\ve_1}|\om_{\ve}]$ as
\[
[\om_U|\om_{\ve}]:=[\om_{\ve_n}|\cdots
|\om_{\ve_1}|\om_{\ve}].
\] 

We suppose now that 
\[U=U_1X_0 \qquad \qquad \text{with }|U_1|\geqs 1.
\]

We have a map from 
$\decx(U,X_0)$ to $\decx(U_1,X_0)$ that sends a decomposition 
\[(\mb V, \mb k)=
\left(
 (V_1, \ldots, V_p),(k_1, \ldots, k_p)
\right)
\] 
to 
\[
\left\{ 
\begin{array}{ll}
(\mb V, (k_1, \ldots , k_p -1))\qquad~ & \mx{if } k_p \neq 0 \\
((V_1, \ldots, V_p'), \mb k ) & \mx{if } k_p = 0 \mx{ and }V_p=V_p'X_0.
\end{array}
\right.
\]
Any decomposition $(\mb {V'}, \mb{k'})$ in  $\decx(U_1,X_0)$ has exactly two
preimages by this map. If one writes 
\[
U_1=V_1'X_0^{k_1'}\cdots V_1'X_0^{k_p'}
\]
then it leads to two decompositions of $U$ 
\[
((V_1',\ldots , V_p'), (k_1', \ldots, k_p'+1))
\qquad \mx{and} \qquad
((V_1',\ldots , V_p',X_0), (k_1', \ldots, k_p',0)).
\]

By induction we have in $V_{n-1}(\m_{0,4})$
\[
[
(\tau^2)^*(\om_{\ve_n})
| \cdots | (\tau^2)^*(\om_{\ve_2})]=
\sum_{\substack{(\mb{V'}, \mb{k'}) \in \\ \decx(U_1,X_0)}} 
(-1)^{|\mb{ V'}|}
\om_{X_0^{|V_1'|} X_1^{k_1'}\cdots
  X_0^{|V_p'|}X_1^{k_p'}}.
\]
We deduce from the previous equality and using the linearity of the tensor
product that $[(\tau^2)^*(\om_{\ve_n}) | \cdots | (\tau^2)^*(\om_{\ve_1})]$
is equal to  
\[
\sum_{\substack{(\mb{V'}, \mb{k'}) \in \\ \decx(U_1,X_0)}} 
(-1)^{|\mb{ V'}|}
[\om_{X_0^{|V_1'|} X_1^{k_1'}\cdots
  X_0^{|V_p'|}X_1^{k_p'}}|-\om_0+\om_1].
\]
This sum can be decomposed as 
\begin{multline*}
\sum_{\substack{(\mb{V'}, \mb{k'}) \in \\ \decx(U_1,X_0)}} 
(-1)^{|\mb{ V'}|+1}
[\om_{X_0^{|V_1'|} X_1^{k_1'}\cdots
  X_0^{|V_p'|}X_1^{k_p'}}|\om_0] + \\
\sum_{\substack{(\mb{V'}, \mb{k'}) \in \\ \decx(U_1,X_0)}} 
(-1)^{|\mb{ V'}|}
[\om_{X_0^{|V_1'|} X_1^{k_1'}\cdots
  X_0^{|V_p'|}X_1^{k_p'}}|\om_1].
\end{multline*}
The first term of the sum is equal to 
\[
\sum_{\substack{(\mb{V'}, \mb{k'}) \in \\ \decx(U_1,X_0)}} 
(-1)^{|\mb{ V'}|+1}
\om_{X_0^{|V_1'|} X_1^{k_1'}\cdots
  X_0^{|V_p'|}X_1^{k_p'}X_0} 
\]
and the second term is equal to 
\[
\sum_{\substack{(\mb{V'}, \mb{k'}) \in \\ \decx(U_1,X_0)}} 
(-1)^{|\mb{ V'}|}
\om_{X_0^{|V_1'|} X_1^{k_1'}\cdots
  X_0^{|V_p'|}X_1^{k_p'+1}} .
\]
 
The previous discussion on $\decx(U,X_0)$ tells us that adding the two sums  
above gives
\[
\sum_{(\mb V, \mb k) \in \decx(U,X_0)} 
(-1)^{|\mb V|}
\om_{X_0^{|V_1|} X_1^{k_1}X_0^{| V_2|} X_1^{k_2}\cdots
  X_0^{|V_p|}X_1^{k_p}} .
\] 
This gives \eqref{IIitip} when $U=U_1X_0$.

If $U=U_1X_1$ with $|U_1|\geqs 1$, we have a one to one correspondence between
$\decx(U_1,X_0)$ and $\decx(U,X_0)$ defined by 
\[
((V_1', \ldots, V_p'),(k_1',
\ldots, k_p'))\mapsto \left\{
\begin{array}{ll}
((V_1', \ldots, V_p',X_1),(k_1',\ldots, k_p',0))& \mx{if }k_p'>0\\
((V_1', \ldots, V_p'X_1),(k_1',\ldots, k_p'))&\mx{otherwise.}
\end{array} \right.
\] 
Then
\eqref{IIitip} follows by induction using the linearity of the tensor product. 
\end{proof}
\subsection{Bar Construction on $\m_{0,5}$ and the pentagon relations}
Here, we will show how the pentagon relations \eqref{III}  are related
to the bar construction on $\m_{0,5}$. 

The shuffle algebra $B(\m_{0,5})$ being much more complicated than
$B(\m_{0,4})$ we will first review some facts explained in
\cite{BrownMZVPMS}. We now fix a dihedral structure $\delta$, as described
in \cite{BrownMZVPMS}, on $\m_{0,5}$. We will used the ``standard'' dihedral
structure given by ``cyclic'' order on the marked points 
\[
z_1<z_2<z_3<z_4<z_5(<z_1) 
\] 
or with our normalization 
\[
0<xy<y<1<\infty.
\] 
This corresponds to a good choice of
coordinates to study the connected components of $\m_{0,5}(\R)$ such
that the marked points are in the order given by $\delta$. We will refer
to that component as the standard cell. 

More precisely, let $i$, $j$, $k$, $l$ denote distinct elements of $\{1,2,3,4,5\}$. The
cross-ratio $[i\,j|k\,l]$ is defined by the formula:
\[
[i\,j|k\,l]=\frac{z_i-z_k}{z_i-z_l}\frac{z_j-z_l}{z_j-z_k}.
\] 
Brown, in \cite[Sections 2.1 and 2.2]{BrownMZVPMS}, has defined coordinates on
$\m_{0,n}$, and more generally on an open $\m_{0,n}^{\delta}$ of the
Deligne-Mumford compactification of the moduli space of curves $\ol{\m_{0,n}}$, such that
\[
\m_{0,n}\subset\m_{0,n}^{\delta} \subset \ol{\m_{0,n}} .
\] 
These
coordinates respect the natural dihedral 
symmetry of the moduli spaces of curves. Applying his work to the case $n=5$,
let $i$ and $j$ be in 
$\{1,2,3,4,5\}$ such that $i$, $i+1$, $j$ and $j+1$ are distinct. We set
\[
u_{ij}=[i\,i+1|j+1\,j].
\]
In particular, the codimension $1$ components of $\partial \m_{0,n}$ contained in
$\m_{0,n}^{\delta}$ are given 
by $u_{ij}=0$ and the standard cell is contained in $\m_{0,n}^{\delta}$.

The coordinates $u_{ij}$ satisfy the relations
\begin{equation}\label{relcoord}
u_{ij}u_{im}+u_{kl}=1,
\end{equation}
for $k\equiv i-1 \mod 5$, $l \equiv i+1 \mod 5$
 and $m\equiv j+1 \mod 5$, all the indices being in $\{1,2,3,4,5\}$ and such
 that $u_{ij}$, $u_{im}$ and $u_{kl}$ are defined.

 In \cite{BrownMZVPMS}[Corollary 2.3],
these relations are given in terms of two sets of chords 
of a polygon (a pentagon for $\m_{0,5}$) and the picture corresponding to the
above relation is given below.
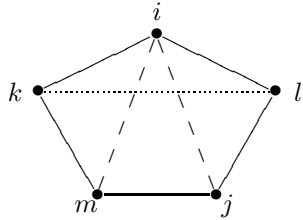
\[
\xymatrix@W=0ex@M=0ex@R=4ex@C=4ex{
 & & \bullet  \save[]+<0ex,2ex>*{i} \restore \ar@{--}[dddr] \ar@{--}[dddl]
\ar@{-}[drr]\ar@{-}[dll]& & \\
\bullet\save[]+<-2ex,0ex>*{k} \restore\ar@{-}[ddr] & & & &
 \bullet \save[]+<2ex,0ex>*{l}\restore \ar@{-}[ddl] \ar@{.}[llll] 
\\
& & & & \\
& \bullet \save[]+<-1ex,-1ex>*{m}\restore \ar@{-}[rr] & &
 \bullet \save[]+<1ex,-1ex>*{j} \restore
}
\]

Let $\om_{12}$, $\om_{23}$, $\om_{34}$, 
$\om_{45}$, $\om_{24}$ be the differential forms
\begin{multline*}
\om_{12} =\dd\log(u_{25})=\frac{\dx}{x}, \quad 
\om_{23}=\dd\log(u_{31}u_{41})=\frac{\dx}{x-1}, \quad \\
\om_{34}=\dd\log(u_{24}u_{41})=\frac{\dy}{y-1}, \quad 
\om_{45}=\dd \log(u_{35})=\frac{\dy}{y} \quad \\
 \text{and} \quad
\om_{24}=\dd \log(u_{41})=\frac{\dd(xy)}{xy-1}. 
\end{multline*}

If $W$ is a word in $\axb=\{ X_{34}, X_{45}, X_{24},
X_{12}, X_{23}\}^*$ with $|W|=n$, we will write $\om_W \in
\Hdr^1(\m_{0,5})^{\otimes n}$ for the bar symbol $[\om_{i_nj_n} | \cdots
|\om_{i_1j_1} ]$. Note that  the elements $\om_W$ for $W$ in $\axb$ are not 
all in $V(\m_{0,5})$; in general, only linear combinations of such symbols are 
in $V(\m_{0,5})$. 
\begin{exm}
The elements $[\om_{12}]$, $[\om_{23}]$ and $[\om_{12}|\om_{23}]$ are in
$V(\m_{0,5})$ even if  $[\om_{12}|\om_{45}]$ is not. However  $[\om_{12}|\om_{45}] +
[\om_{45}|\om_{12}]$ is in $V(\m_{0,5})$.
\end{exm}

Example 3.43 in \cite{BrownMZVPMS} (using  \cite[Thm. 3.38 and
Coro. 3.41]{BrownMZVPMS})
tells us that the exact sequence
\[
0 \lra  \C \llan X_{24} , X_{34},X_{45} \rran \lra \UB \lra  \C\llan
X_{12,}X_{23} \rran \lra 0
\]
is dual to the exact sequence
\[
0 \lra V(\m_{0,4}) \lra V(\m_{0,5}) \lra \C\langle \frac{\dd y}{y},  \frac{\dd
  y}{y-1}, \frac{x\dd y}{xy-1}\rangle  \lra 0
\]
which comes from the expression, in cubical coordinates, of the map $\m_{0,5}\lra \m_{0,4}$
which forgets the $4^{Th}$ point. Thus, the identification 
\[
\UB \simeq \C \llan X_{24} , X_{34},X_{45} \rran \rtimes \C\llan
X_{12,}X_{23} \rran
\]
is dual (as graded algebra) to 
\[
V(\m_{0,5}) \simeq V(\m_{0,4})  \otimes  \C \langle \frac{\dd y}{y},  \frac{\dd
  y}{y-1}, \frac{x\dd y}{xy-1}\rangle 
\]
and $V(\m_{0,5})$ is the graded dual $\UB^*$ of  $\UB$.

The graded  dual of the free non-commutative algebra of formal series 
\[
R=\C\llan X_{34}, X_{45}, X_{24},
X_{12}, X_{23}\rran 
\]
 is the shuffle algebra 
\[
T
:=\bigoplus_n \left( \C\om_{34}\oplus \C\om_{45}\oplus \C\om_{24}\oplus
\C\om_{12}\oplus \C\om_{23} \right)^{\otimes n}.
\]

Let $\Om$ be the element in $R \otimes \Hdr^1(\m_{0,5})$ defined by 
\[
\Om=X_{12}\otimes \om_{12} +X_{23}\otimes \om_{23} +X_{34}\otimes \om_{34}
+X_{45}\otimes \om_{45} +X_{24}\otimes \om_{24}.  
\] 
and
\[
\Exp(\Om):=\sum_{W\in \axb} W \otimes \om_{W} \in R \otimes T.
\] 
The element $\Exp(\Om)$ corresponds to the identity of $R$ and encodes the fact
that the dual of a word $W$ is $\om_{W}$. 
A word $W$ (seen in $\UB$) is written in the basis $B_4$ as
\[
W =\sum_{b_4\in B_4} l_{b_4,W} b_4.
\] 
Duality between $R$ and $T$ and between $\UB$ and $V(\m_{0,5})$ tells us
that, the basis $B_4^*=(b_4^*)_{b_4 \in B_4}$ of $V(\m_{0,5})$ dual to $B_4$ is given by
\[
\forall b_4 \in B_4 \qquad b_4^*=\sum_{W\in \axb} l_{b_4,W} \om_{W}.
\]

Using the projection $R \ra \UB$ one
can see $\Exp(\Om)$ in $\UB \otimes T$. Actually, by duality, $\Exp(\Om)$ lies
in $\UB \otimes V(\m_{0,5})$.  So, writing each $W$ in the basis $B_4$ leads to
the following expression of  
$\Exp(\Om)$ in $\UB\otimes V(\m_{0,5})$ 
\[
\Exp(\Om)=\sum_{b_4 \in B_4} b_4 \otimes b_4^*\quad \in  \UB
\otimes V(\m_{0,5}).
\] 
Thus, $\Exp(\Om)$ realized the identification between the graded dual of $\UB$
and $V(\m_{0,5})$ as was observed by Furusho in \cite{DSPfuru0}. This
discussion can be summarized by the following proposition.

\begin{prop}
We have a natural identification 
\[
\UB^* \simeq V(\m_{0,5}),
\]
$\UB^*$ being the graded dual of $\UB$. 

This identification gives a basis
$B_4^*$ of $V(\m_{0,5})$ dual to $B_4$ the basis of $\UB$ which comes from the
identification 
\[
\UB \simeq \C \llan X_{24} , X_{34},X_{45} \rran \rtimes \C\llan
X_{12,}X_{23} \rran.
\]
The basis $B_4^*=(b_4^*)_{b_4 \in B_4}$ is explicitly given for all $b_4$ in the
basis $B_4$ by 
\begin{equation}\label{baseB4*}
b_4^*=\sum_{W \in \axb} l_{b_4,W} \, \om_{W}.
\end{equation}
\end{prop}

Let $\wh{\m_{0,5}}$ be the universal covering of $\m_{0,5}$. A
multi-valued function on $\m_{0,5}$ is an analytic function on
$\wh{\m_{0,5}}$. 
Consider the formal differential equation on $\wh{\m_{0;5}}$
\[
\dd L = \Om L
\]
where $L$ takes values in $\UB$, whose coefficients are multi-valued
functions on $\m_{0,5}$. As in the case of the equation \eqref{KZ}, if
we fix either the value of $L$ at some point of $\m_{0,5}$ or its asymptotic
behavior at a tangential base point, then the solution is unique.

The irreducible components of codimension $1$ of $\partial\m_{0,5}$ in
$\ol{\m_{0,5}}$ are in one to one 
correspondence with the $2$-partitions of $\{z_1, z_2,z_3, z_4,z_5\}$ and
will be denoted as $z_{i_1}z_{i_2}|z_{i_3}z_{i_4}z_{i_5}$. These boundary
components are all isomorphic to $\ol{\m_{0,4}}$. Here, we will
only consider the following components $\bar D_{52}=z_1z_2|z_3z_4z_5$,
$\bar D_{13}=z_2z_3|z_4z_5z_1$, $\bar D_{24}=z_3z_4|z_5z_1z_2$,
$\bar D_{35}=z_4z_5|z_1z_2z_3$, $ \bar D_{41}=z_5z_1|z_2z_3z_4$ (we may use the
convention $\bar D_{ij}=\bar D_{ji}$).  One remarks that
those components are given by a partition that respect the dihedral
structure $\delta$ and the numbering $ \bar D_{ij}$ is coherent with the
notation of \cite{BrownMZVPMS}. We will write $D_{ij}\simeq
\m_{0,4}^{\delta}$ for the 
intersection of $\bar D_{ij}$ with $\m_{0,5}^{\delta}$
. The divisors $D_{ij}$ are given in the
dihedral coordinates by $u_{ij}=0$. Following  Brown, we have $5$
tangential base points (corresponding to the intersection of $2$
irreducible components) given by the triangulation of the polygon
corresponding to $\delta$; as we are working in $\m_{0,5}$, the polygon is a
pentagon, and a triangulation is given by two chords going out from a
single vertex, so one can number the triangulation by the number of its
vertex: precisely, one has 
\begin{align*}
P_3&=D_{35}\cap D_{13}, \qquad P_1=D_{13}\cap D_{41}, \qquad
P_4=D_{41}\cap D_{24}, 
\\ 
P_2&=D_{24}\cap D_{52}, \quad \text{ and }\quad
P_5=D_{52}\cap D_{35}.
\end{align*}
Let $L_{i}$  be the normalized solution at $P_i$ (see
\cite{BrownMZVPMS} Theorem 6.12).

Now, we fix a basis $B=(b)_{b \in B}$ of $\UB$ and its dual basis
$B^*=(b^*)_{b \in B}$ in $V(\m_{0,5})$. The description of the situation in
dimension $1$ and section 5.2 in \cite{BrownMZVPMS} shows that Theorem 6.27
of Brown's article in \cite{BrownMZVPMS} can be rewritten as follows.
\begin{prop}
For any tangential base point $P_i$, one can write
$L_i(z)$ as  
\[
\forall z \in \wh{\m_{0,5}} \quad
 L_i(z)=\sum_{b\in B} (\iti_{\gamma} b^*)b
\]
where $\gamma$ is a path from $P_i$ to $z$ and where iterated
integrals are regularized iterated integrals.
\end{prop}

The comparison of two different normalized solutions at two different
base points $P_i$ and $P_j$ is then given by 
\[
\forall z\in \wh{\m_{0,5}} \qquad L_i(z)=L_j(z)\left( \sum_{b\in B}
  \left(\iti_{\gamma} b^*\right)b\right) 
\] 
where $\gamma$  is any path going from $P_i$ to $P_j$ homotopically equivalent to
 a path $\gamma'$ going from $P_i$ to $P_j$ in the standard cell of $\m_{0,5}(\R)$.

Brown shows how to restrict any element $\om$ in $B(\m_{0,5})$ to any
boundary components $D$ introducing a regularization map $\on{Reg}(\om,D)$.
This map sends
each $\frac{\du_{ij}}{u_{ij}}$ to $0$ if the restriction of $u_{ij}$ to $D$
equals $0$ or $1$. More precisely,
\begin{defn}
Let $D_{ij}$ be a boundary component of $\m_{0,n}$ given by $u_{ij}=0$. We
define $\Reg(\frac{\du_{kl}}{u_{kl}},D_{ij})$ as follows:
\begin{itemize}
\item $\Reg(\frac{\du_{ij}}{u_{ij}},D_{ij})=0$,
\item $\Reg(\frac{\du_{kl}}{u_{kl}},D_{ij})=0$ if $u_{ij}u_{im}+u_{kl}=1$ as in
  \eqref{relcoord}, 
\item $\Reg(\frac{\du_{kl}}{u_{kl}},D_{ij})=\frac{\dd \tilde{u_{kl}}}{
    \tilde{u_{kl}}}$ where $\tilde{u_{kl}}$ is the restriction of $u_{kl}$ to
  $D_{ij}$ using the natural inclusion
\[
 D_{ij} \hookrightarrow \m_{0,5}^{\delta}.
\] 
\end{itemize}
\end{defn}
Now, using the inclusion $
\m_{0,4} \hookrightarrow \m_{0,4}^{\delta} \simeq D_{ij},
$ one can define  the map 
\[\Reg(-,D_{ij}) : V(\m_{0,5}) \lra V(\m_{0,4}).\]
 It sends an element 
\[
\om=\sum c_{i_1,j_1, \ldots, i_k,j_k} [\om_{i_1j_1}|\ldots|\om_{i_kj_k}] \in V(\m_{0,5})
\]
to
\[
\Reg(\om,D_{ij})=\sum c_{i_1,j_1, \ldots, i_k,j_k} 
[
\Reg(\om_{i_1j_1},D_{ij})|\ldots|\Reg(\om_{i_kj_k},D_{ij})] \in V(\m_{0,4}).
\]

\begin{exm}\label{exmrestrict}
As explained in Brown \cite[Lemma 2.6]{BrownMZVPMS}, the restriction of the
coordinate $u_{25}$ on $D_{35}$  can be computed in terms of the dihedral
coordinates on $D_{35} \simeq \m_{0,4}^{\delta}$ as follows. The chord $(3,5)$ splits the
pentagon
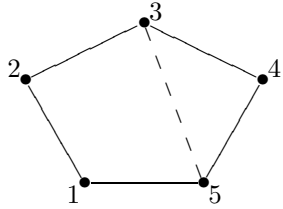
\[
\xymatrix@W=0ex@M=0ex@R=4ex@C=4ex{
 & & \bullet  \save[]+<1ex,1ex>*{3} \restore \ar@{--}[dddr]
\ar@{-}[drr]\ar@{-}[dll]& & \\
\bullet \ar@{-}[ddr]  \save[]+<-1ex,1ex>*{2} \restore& & & & 
\bullet  \save[]+<1ex,1ex>*{4} \restore \ar@{-}[ddl]\\
& & & & \\
& \bullet \ar@{-}[rr]  \save[]+<-1ex,-1ex>*{1} \restore& &
 \bullet \save[]+<1ex,-1ex>*{5} \restore
}
\]
into a square and a triangle
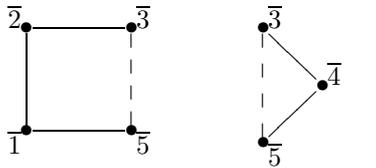
\[
\xymatrix@W=0ex@M=0ex@R=4ex@C=4ex{
\bullet  \save[]+<-1ex,1ex>*{\ol 2} \restore \ar@{-}[rr] \ar@{-}[dd] & & 
\bullet  \save[]+<1ex,1ex>*{\ol 3} \restore \ar@{--}[dd] \\
 & & \\
\bullet \ar@{-}[rr]  \save[]+<-1ex,-1ex>*{\ol 1} \restore& &
 \bullet \save[]+<1ex,-1ex>*{\ol 5} \restore
}
\qquad \qquad
\xymatrix@W=0ex@M=0ex@R=4ex@C=4ex{
\bullet  \save[]+<1ex,1ex>*{\ol 3} \restore \ar@{--}[dd] \ar@{-}[dr]\\
 & \bullet  \save[]+<1ex,1ex>*{\ol 4} \restore & \\
\bullet \ar@{-}[ur] \save[]+<1ex,-1ex>*{\ol 5} \restore & 
},
\]
where we have written $\ol i$ instead of $i$ to keep track of the difference
between the labeling on the pentagon (corresponding to $\m_{0,5}^{\delta}$) and the
square  (corresponding to $\m_{0,4}^{\delta}\simeq D_{35}$).
 
This decomposition corresponds to the isomorphism
\[
D_{35}\st{\sim}{\lra} \m_{0,4}^{\delta}\times \m_{0,3}^{\delta} \st{\sim}{\lra}
\m_{0,4}^{\delta}
\]
where, in $\m_{0,4} \subset \m_{0,4}^{\delta}$, the four marked points  are
labeled $z_{\ol 1}$, $z_{\ol 2}$, $z_{\ol 3}$ and $z_{\ol 5}$.

The coordinate $u_{25}$ is given by the cross-ratio
\[
u_{25}=[23|15].
\]
Its restriction to $D_{35}$ is the coordinate given by the chord
$(\ol 2, \ol 5)$ and thus by the cross-ratio
\[
\tilde{u_{25}}=[\ol 2 \ol 3|\ol 1 \ol 5].
\]

Following this description, there are two dihedral coordinates on $D_{35}\simeq
\m_{0,4}^{\delta}$ given by
\[
t_1=[\ol 2 \ol 3|\ol 1 \ol 5] \qquad \mx{and} \qquad
t_2=[\ol 1 \ol 2|\ol 3 \ol 5].
\]

Similarly, $u_{13}$, corresponding to the chords $(1,3)$ in the pentagon
description, restricts on $D_{35}$ to $t_2$ which corresponds to the chord 
$(\ol 1, \ol 3)$ on the square description of $D_{35}$. As $P_5$ is defined by
$u_{25}=u_{35}=0$, one sees that $t_1=\tilde{u_{25}}$ is $0$ at $P_5$ and
similarly that $t_2=\tilde{u_{13}}$ is $0$ at $P_3$. Moreover, on $D_{35}$
one has $t_2=1-t_1$, which agrees with the fact that on $\m_{0,5}$ one has
$u_{25}+u_{13}u_{14}=1$ and $u_{14}+u_{25}u_{35}=1$. Thus, the coordinate $t_1$
is equal to $1$ at $P_3$ and $t_2$ 
is equal to $1$ at $P_5$.
\end{exm} 
\begin{prop}\label{REGIII}
For any two consecutive tangential base points $P_i$ and $P_j$ with
$j\equiv i-2 \mod 5$, one has
\[
\forall z \in \wh{\m_{0,5}}  \qquad L_i(z)=L_j(z)\left( \sum_{b\in B}
  \left(\iti_{p_{ji}} \on{Reg}(b^*,D_{ji})\right)b\right) 
\]
where $p_{ji}$ is the real segment going in $D_{ji}$ from $P_i$ to
$P_j$. 
\end{prop}
\begin{proof}
The symmetry of the situation allows us to prove it only in the case where
$i=5$, $j=3$ and $B$ is the basis $B_4$. 

Let $p_{35}$ be the path in
$D_{35}$ going from $P_5$ to $P_3$; we need to show that 
\begin{equation}\label{compregonDij}
L_3(z)^{-1}L_5(z)= \sum_{b_4 \in B_4 }
  \left(\iti_{p_{35}}\Reg( b_4^*,D_{35}) \right)b_4.
\end{equation}

Brown, in \cite[Definition 6.18]{BrownMZVPMS}, defined $Z^{35}$ to be the
quotient $L_3(z)^{-1}L_5(z)$. Using the proof of Theorem
6.20 in \cite{BrownMZVPMS}, we have 
\[
Z^{35}=L_3(z)^{-1}L_5(z)= 
\sum_{
\substack{
W=X_{i_nj_n} \cdots X_{i_1j_1}\\
 \in \{X_{12}, X_{23}\}^* } }
  \left(\iti_{p} \frac{\dt}{t-\ve_n}\w \cdots \w\frac{\dt}{t-\ve_1}
\right)W  
\]
with $\ve_k=0$ if $i_k=1$ (and $j_k=2)$) and $\ve_k=1$ otherwise (that is,
$i_k=2$ and $j_k=3$). Using the morphism $p_4 : \UB \lra \UF$ that send $X_{i4}$
to $0$, $X_{12}$ to $X_0$ and $X_{23}$ to $X_1$, we have:
\[
Z^{35}=L_3(z)^{-1}L_5(z)= 
\sum_{
\substack{
W=X_{i_nj_n} \cdots X_{i_1j_1}\\
 \in \{X_{12}, X_{23}\}^* } }
  \left(\iti_{p} \om_{p_4(W)}
\right)W.
\]

We recall that an element $b_4$ of the basis $B_4$ is either $1$ or a monomial
of the form
\begin{equation}\label{baseB4}
b_4=U_{245}V_{123} \qquad U_{245}\in \{X_{24},X_{34},X_{45}\}^*, \quad
V_{123}\in\{X_{12},X_{23}\}^*.
\end{equation}

So, in order to prove \eqref{compregonDij}, it is enough to prove that:
\begin{itemize}
\item All the
iterated integrals $\iti_{p_{35}}\Reg( b_4^*,D_{35})$ for $b_4 = U_{245}V_{123}$
with $U_{245}$ not empty vanish:
\begin{equation*}
 \left.
\begin{array}{c} b_4 =
U_{245}V_{123} \\
\text{ with } U_{245}\in \{X_{24},X_{34},X_{45}\}^*, \quad U_{245}\neq
\emptyset
\end{array}\right\} \quad \Rightarrow \quad
\iti_{p_{35}}\Reg( b_4^*,D_{35})=0.
\end{equation*}
\item All the
iterated integrals $\iti_{p_{35}}\Reg( b_4^*,D_{35})$ for $b_4 = V_{123}$ are
equal to 
\[
\iti_{p}\om_{p_4(V_{123})}=\iti_{p}\om_{p_4(b_4)}.
\]
That is:
\[
b_4 =V_{123}\in \{X_{12},X_{23}\}^*
\quad \Rightarrow 
\quad 
\iti_{p_{35}}\Reg( b_4^*,D_{35})=\iti_{p}\om_{p_4(b_4)} .
\]
\end{itemize}

Let $t$ denote the dihedral coordinate $t_1$ on $D_{35}$ which takes
values $0$ at $P_5$ and $1$ at $P_3$ (see Example \ref{exmrestrict}). Example
\ref{exmrestrict} shows that 
\[
\tilde{u_{25}}=t,\qquad  \tilde{u_{13}}=1-t.
\] 
Moreover, as 
\[
u_{24}+u_{13}u_{35}=1 \qquad \mx{and} \qquad u_{14}+u_{25}u_{35}=1,
\]
one has $u_{24}=u_{14}=1$ on $D_{35}$. 

As the differential forms $\om_{23}$ and $\om_{34}$ are defined by
\begin{align*}\om_{23}=&\dd\log(u_{31}u_{41})=\dd \log(u_{13})+\dd \log(u_{14})
  \qquad \mx{and} \\
\om_{34}=&\dd\log(u_{24}u_{41})=\dd \log(u_{24})+\dd \log(u_{14}),
\end{align*}
and since one has
\[
\Reg(\dd\log(u_{35}),D_{35})=\Reg(\dd\log(u_{24}),D_{35})=\Reg(\dd\log(u_{14}),D_{35})=0,
\]
one concludes that 
\begin{align*}
\Reg(\om_{12},D_{35})=&\Reg(\dd\log(u_{25}),D_{35})=\frac{\dt}{t},\\ 
\Reg(\om_{23},D_{35})=&\Reg(\dd\log(u_{13}),D_{35})=\frac{\dt}{t-1}
\end{align*}

and $\Reg(\om_{ij},D_{35})=0$ otherwise.

It is now enough to show that for $b_4$ in $B_4$
\begin{itemize}
\item  $b_4$ is a word in the letters $X_{12}$ and $X_{23}$ (that is  $b_4 \in
  \{X_{12},X_{23}\}^*$) if and only if 
\begin{align*}
b_4^*&=\om_{b_4} \quad  \text{with } b_4\in  \{X_{12},X_{23}\}^* \\
 (&=[\om_{i_nj_n}|\cdots|\om_{i_1j_1}] \quad \text{with }
X_{i_kj_k} \in  \{X_{12},X_{23}\})
\end{align*}
\item $b_4$ contains some $X_{ij}$ with $i=4$ or $j=4$ if and only if
\begin{align*}
b_4^*&=\sum \lambda_{W'}\om_{W'} \quad \mx{with} \quad 
\lambda_{W'}\neq 0 \Rightarrow W' \notin  \{X_{12},X_{23}\}^* 
\end{align*}
that is, if and only if $b_4^*$ is a linear combination of bar symbols $\sum
\lambda_{W'}\om_{W'}$ ($\lambda_{W'}\neq 0$) with $W'$ containing 
at least one of the letters $X_{34},X_{45},X_{24}$. 
\end{itemize}
Using equations \eqref{baseB4}  and \eqref{baseB4*} that describe
respectively $b_4$ and $b_4^*$, one sees that Equation \eqref{compregonDij} (and thus the
proposition) follows directly from the relation defining $\UB$.
\end{proof}
From the previous proposition, we immediately deduce the following corollary.
\begin{coro}\label{corohomo}
For any path $\gamma$ in the standard cell homotopically equivalent to $p_{ji}$
$j\equiv i-2\mod 5$ ($1\leqs i,j \leqs 5$), we have  
\[
\forall \om \in V(\m_{0,5})\qquad \iti_{\gamma}\om=\iti_{p_{ji}}\Reg(\om,D_{ji}).
\]
\end{coro}
Let $\gamma=p_{35} \circ p_{52} \circ p_{24} \circ p_{41} \circ
p_{13}$ denote the  composed path beginning and ending at $P_3$ and extending
the map $\Reg(\om, \gamma)$  to paths that are piecewise in some of the
divisor $D_{ij}$.
\begin{thm}\label{IIIMZVgeom}
The relation \eqref{III} is equivalent to the family of relations
\[
\forall b_4 \in B_4,\, b_4\neq 1 \qquad \iti_{\gamma}\Reg(b_{4}^*,\gamma)=0
\]
which is exactly the family \eqref{IIIMZV}.
\end{thm}
\begin{proof}For $i$ in $\{1,2,3,4,5\}$ and $j=i-2 \mod 5$, we define $Z^{ji}$
  by the formula
\[
Z^{ji}=\left( 
\sum_{b_4\in B_4}
  \left(\iti_{p_{ji}} \on{Reg}(b_4^*,D_{ji})\right)
b_4\right)\!.
\]
By Proposition \ref{REGIII}, one has 
\[
\forall z \in \wh{\m_{0,5}} \qquad L_i(z)=L_j(z)Z^{ji}.
\]

Comparison between the $5$ normalized solutions $L_i$ at the $5$ tangential 
base points $P_i$ gives  
\begin{equation}\label{5cycleZij}
\forall z \in \wh{\m_{0,5}}\qquad  
L_{3}(z)=L_{3}(z)Z^{35}Z^{52}Z^{24}Z^{41}Z^{13}.
\end{equation}

In the proof of Theorem 6.20
\cite{BrownMZVPMS} and the example which follows it, Brown proves that the
product of the $Z^{ji}$ is equal to the L.H.S (that is the product of the
$\pkz$) of \eqref{III}. So, 
Equation \eqref{III} can be written as
\[
Z^{35}Z^{52}Z^{24}Z^{41}Z^{13}=1.
\] 
It can also be proved directly using Proposition \ref{REGIII}. 

Equation \eqref{5cycleZij} is given by the analytic
continuation of the solution $L_3$ along any path in the standard cell beginning
and ending at $P_3$ and going through $P_1$, $P_4$, $P_2$ and
$P_5$ (in that order). Such a path is homotopically equivalent to $\gamma$ (and
to $0$) and
the product of the $Z^{ji}$ can be written as 
\[
Z^{35}Z^{52}Z^{24}Z^{41}Z^{13}=\sum_{b_4\in
  B_4}\left(\iti_{\gamma}b_4^*\right)b_4. 
\]

As $\gamma$ is  homotopically equivalent
to $0$, each of the homotopy invariant regularized iterated integrals
above are $0$ (except for $b_4=1$). Thus, the product 
\[
Z^{35}Z^{52}Z^{24}Z^{41}Z^{13}
\]
is equal to $1$. We deduce from the previous discussion that the family of
relations 
\[
\forall b_4 \in B_4,\, b_4\neq 1 \qquad \iti_{\gamma}\Reg(b_{4}^*,\gamma)=0
\]
implies relation \eqref{III}. Moreover, one deduces from the equation 
\[
Z^{35}Z^{52}Z^{24}Z^{41}Z^{13}=\sum_{b_4\in
  B_4}\left(\iti_{\gamma}b_4^*\right)b_4. 
\]
that relation \eqref{III} (which says that the product of the $Z^{ji}$
is $1$) implies 
\[
\forall b_4 \in B_4,\, b_4\neq 1 \qquad \iti_{\gamma}\Reg(b_{4}^*,\gamma)=0.
\]
The first part of the theorem is then proved.

Using the expression of $b_4^*$ in terms of $\om_{W}$, the end of the
theorem follows from Proposition \ref{propitiC5W} below.
\end{proof}
From the previous theorem, one deduces the following corollary.
\begin{coro}
For any basis $B$ of $\UB$, the pentagon relation
$\eqref{III}$ is equivalent to 
\[
\forall b \in B \qquad \iti_{\gamma}\Reg(b^*,\gamma)=0
\]
where $\gamma$, as previously, is the path $p_{35} \circ p_{52} \circ
p_{24} \circ p_{41} \circ p_{13}$.
\end{coro} 
Following the proof of \ref{REGIII}, one proves Proposition \ref{propitiC5W},
which completes the proof of Theorem \ref{IIIMZVgeom}. 
\begin{prop}\label{propitiC5W}
For any bar symbol $\om_W$ dual to a word $W$ in the letters $X_{34}$,
$X_{45}$, $X_{24}$, $X_{12}$, $X_{23}$,
we have
\[
C_{5,W,KZ}=\iti_{\gamma} \Reg(\om_{W},\gamma)
\]
where $C_{5,W,KZ}$ is the coefficient $C_{5,W}$ defined in \eqref{C5Wdef} in the
particular case of the Drinfel'd associator $\pkz$.
\end{prop}
\begin{proof}
To show the proposition, it is enough, using the decomposition of
$\gamma=p_{35} \circ p_{52} \circ p_{24} \circ p_{41} \circ
p_{13}$, to show that for any $U$ in
$\AXij$ and any $i$,
one has 
\[
(-1)^{\degg(U)}\zsha(\pg(U))=\iti_{I_i}\Reg(\om_{U},I_i)
\] 
where $I_5=p_{13}$, $I_4=p_{41}$,
$I_3=p_{24}$, $I_2=p_{52}$ and $I_1=p_{35}$. 

As
$\Reg(\om_{kl}, I_i)= \om_{\pg(X_{kl})}$, the proposition follows.
\end{proof}
\section{Appendix : relations in low degrees}
\subsection{Remarks}
From the following tables, one can see that coefficients of words in $X_{12}$ and 
$X_{23}$ yield the family of relations \eqref{rel1coefeq} (which is equivalent to
\eqref{rel1}). This can be proved 
directly from \eqref{rel3coefeqB4} (which is equivalent to \eqref{rel3}). In order
to do so, one 
observes that if $b_4$ in the 
basis $B_4$ is a word in $X_{12}$ and $X_{23}$, then $l_{b_4,W}\neq 0$ if and
only if $W=b_4$. In the case of the Drinfel'd associator $\pkz$, only the term
\begin{multline*}
\sum_{
U_1\cdots U_5=b_4}
(-1)^{\dea(U_1)+\deb(U_2) +\dec(U_3)+ \ded(U_4)+ \dee(U_5)} \\
\zsha(\pa(U_1))
\zsha(\pb(U_2))
\zsha(\pc(U_3)
\zsha(\pd(U_4))
\zsha(\pe(U_5))
\end{multline*}
is non zero, and the $U_i$ are words in $X_{12}$ and
$X_{23}$.  
Then, the fact that $\pb(X_{12})=\pb(X_{23})=0$ tells us that
$\pb(U_2)=0$ if $U_2 \neq  \emptyset$. As $\pd(X_{12})=\pe(X_{12})=0$,
$\pd(X_{23})=X_0$ and $\pe(X_{23})=X_1$, we deduce
that $\pd(U_4)$ is $0$ or a power of $X_0$ and $\pe(U_5)$ is $0$ or a power of
$X_1$ (again with $U_4,U_5 \neq \emptyset$). We conclude using the fact that
for $k\geqs 1$, 
\[
\zsha(0)=\zsha(X_0^k)=\zsha(X_1^k)=0.
\]

Using the explicit relations between the coefficients of the associator
\eqref{rel3coefeqB4}, the 
above arguments show the well known  implication
``\eqref{rel3} implies \eqref{rel1}'' proved by Furusho in \cite{FurushoMZVSDA}. 

In \cite{FurushopentaH}, Furusho also proved that \eqref{rel3} implies
\eqref{rel2}. This implication does not appear clearly looking at the
coefficients and
comparing \eqref{rel3coefeqB4} and \eqref{rel2coefeq}. In the case of $\pkz$,
the first reason is that no $\pi$ 
can arise from \eqref{IIIMZV}. In order to see ``\eqref{rel3} implies
\eqref{rel2}'' on the coefficients, one should first replace $(2\pi i)^2$ by
$-24\zsha(X_0X_1)$ 
in \eqref{IIMZV}. The second reason is that the proof of Furusho suggests that
the linear combinations involved are much more complicated than the ones involved
for \eqref{rel3} implies \eqref{rel1} (which is deduced from \eqref{rel3} by sending
$X_{i,4}$ to $0$).

Another set of well-known relations between multiple zeta values are the double
shuffle relations. As the representation of the multiple zeta values with
iterated integrals leads to the quadratic relations
\[
 \zsha(V)\zsha(W)=\sum_{U\in \on{sh}(V,W)}\zsha(U),
\] 
writing the multiple zeta values as series $\zeta(\mb k)=\sum
\frac{1}{n_1^{k_1}\cdots n_p^{k_p}}$ leads to another 
regularization $\zeta^*$ and another set of
quadratic relations (\cite{DMRRac})
 \[
 \zeta^*(\mb k)\zeta^*(\mb l)=\sum_{\mb m \in \on{st}(\mb k,\mb l)}\zeta^*(\mb m)
\] 
where $\on{st}(\mb k,\mb l)$ is a family of tuples of integers defined from $\mb k$ and $\mb
l$ by a combinatorial process. The two regularizations are linked by an
explicit formula, and the set of relations induced by the two set of quadratic
relations is known as \emph{double shuffle relations} (see for example
\cite{DMRRac}).

More recently, in \cite{DSPfuru0}, Furusho proved that \eqref{rel3} implies the
double shuffle relations. Seeing this fact directly on the coefficients is not easy
because one has to find the ``right linear combination''. Although one can give the first
example in weight $3$ (see below), already in degree $4$ one has to look at $211$
relations ... Even looking only at the relations coming from multiplicative
generators of $V(\m_{0,5})$ is difficult. 
 However, Theorem \ref{rel3coef} tells us that no information is lost between 
 relation \eqref{rel3} and the family of relations given by \eqref{rel3coefeqB4}.
Thus, using Furusho's theorem, this family of relations implies the double
shuffle relations.  

Using a more suitable basis to write the relations, one that would give ``nice''
multiplicative generators for $V(\m_{0,5})$, or one coming from a ``simple'' basis
of $V(\m_{0,5})$, may help to progress in the direction of the not known
implication
\begin{center}
``Double shuffle'' implies \eqref{rel3}.
\end{center}
However, this is not certain. A global approach (interpreting the series shuffle
relations as a group-like property as in \cite{DMRRac} or in
\cite{DSPfuru0}) or a geometric approach could be better. 

\begin{exm}In weight $2$, double shuffle relations do not give extra relations
  between multiple zeta values. They tell us the values of the second
  regularization of $\zeta^*(1,1)$: $\zeta^*(1,1)=\zeta(2)/2$, which is different
  from the \emph{shuffle} regularization $\zsha(1,1)=\zsha(Y,Y)=0$.

In weight $3$, the double shuffle relations lead to $\zeta(3,1)=\zeta(2)$, 
which can be written as 
\[
\zsha(X_0X_0X_1)=\zsha(X_0X_1X_1).
\]
This equality is a direct consequence of the duality relation; however, to
recover it from Table \ref{tabledeg3}, one needs to use
$3$ relations.
Indeed, using the coefficients of monomials $X_{45}X_{24}X_{24}$, $X_{24}X_{45}X_{45}$,
$X_{34}X_{45}X_{45}$, one finds
\[
\zsha(X_0X_1X_1)=\zsha(X_1X_0X_0)=\zsha(X_1X_1X_0)=\zsha(X_0X_0X_1).
\] 
\end{exm}

\subsection{Degree $1$, $2$ and $3$}
Here one can find the explicit relations given by the pentagon equation
\eqref{III} in low degree. Writing the product 
\begin{multline*}
\pkz(X_{12},X_{23})
\pkz(X_{34},X_{45})
\pkz(X_{51},X_{12})
\pkz(X_{23},X_{34})\\
\pkz(X_{45},X_{51})=\sum_{b_4} C_{b_4} b_4
\end{multline*} 
in the basis $B_4$, the following tables give the relation
$C_{b_4}=0$ in terms of regularized multiple zeta values.

Let $B_4^{\deg=i}$ denote the family of elements in $B_4$ with degree equal to
$i$. For any $S \subset B_4$, one defines $S^*$ to be the set $\{b^*\, |\, b \in
S\}$. Let $N$ be an integer, $N\geqs 1$. A sequence $\{S_1, \ldots, S_N\}$ with
$S_i \subset B_4^{\deg=i}$ is called 
\emph{a set of multiplicative generators up to degree $N$} if for every $i=1, \ldots, N$
and every element $\om$ of degree $i$ in $V(\m_{0,5})$, $\om$ is a linear
combination of shuffles of elements in $\{1\}\cup S_1^*\cup \cdots S_i^*$. Let $\gamma'$
be a path in the standard cell homotopically equivalent to $\gamma=p_{35} \circ p_{52} \circ
p_{24} \circ p_{41} \circ p_{13}$, and let $f_1$ and $f_2$ be two elements in
$V(\m_{0,5})$. Then it is a property of iterated integrals (\cite{IIDFLSChen})
that
\[
\left(\int_{\gamma'}f_1\right)\left(\int_{\gamma'}f_2\right)
=\left(\int_{\gamma'}f_1\sha f_2\right).
\]
Now, using Corollary \ref{corohomo}, one deduces that
\[
\left(\int_{\gamma}\Reg(f_1,\gamma)\right)\left(\int_{\gamma}\Reg(f_2,\gamma)\right)
=\left(\int_{\gamma}\Reg(f_1\sha f_2,\gamma)\right).
\]
In particular the family of relations 
\[
\forall b_4 \in B_4,\, b_4\neq 1 \qquad \iti_{\gamma}\Reg(b_{4}^*,\gamma)=0
\]
up to degree $N$ is induced by 
\[
\forall i=1, \ldots N, \quad \forall s \in S_i, 
\qquad \iti_{\gamma}\Reg(s^*,\gamma)=0
\]
for any set of multiplicative generators $\{S_1, \ldots, S_N\}$ up to degree $N$. More
precisely, let an element $b_4$ in $B_4$ be of degree less than or equal to $N$. The
corresponding relation between 
multiple zeta values given at Equation \eqref{IIIMZV} is exactly (Cf. Theorem
\ref{IIIMZVgeom}) 
\[
\iti_{\gamma}\Reg(b_{4}^*,\gamma)=0.
\]
Now, we write $b_4^*$ in terms of multiplicative generators 
\[
b_4^*=\sum_{k=1}^M\lambda_k s_{i_k}^* \sha s_{j_k}^*
\]
with $s_{i_k}^*$, $s_{j_k}^*$ in $\{1\}\cup S_1^*\cup \cdots S_N^*$.
Using the previous discussion, one has
\begin{multline*}
\iti_{\gamma}\Reg(b_{4}^*,\gamma)=
\sum_{k=1}^M\lambda_k \iti_{\gamma}
\Reg(s_{i_k}^* \sha s_{j_k}^*,\gamma) \\
=\sum_{k=1}^M\lambda_k 
\left(\iti_{\gamma}\Reg(s_{i_k}^*,\gamma)\right)
\left( \iti_{\gamma}\Reg( s_{j_k}^*,\gamma)\right).
\end{multline*}
Thus, the relation corresponding to $b_4$ is a consequence of the shuffle
relations for the MZV and of the relations corresponding to the $s_{i_k}$ and 
the $s_{j_k}$. 

In degree $1$ the basis is given by the letters $X_{34}$,  $X_{45}$,
$X_{24}$,  $X_{12}$ and  $X_{23}$. The corresponding relations (equivalent to
\eqref{III}) are given in Table \ref{tabledeg1}.

\newcommand{\sui}{=0 \\} 
\newcommand{\zetasha}{\zsha}
\newcommand{\tetecoloneun}{Monomials}
\newcommand{\tetecolonedeux}{Relations}
\newcommand{\findecoeff}{}
In degree $2$ the basis $B_4$ is given by $19$ monomials, but we have only $4$
multiplicative generators and the
corresponding relations are given in Table \ref{tabledeg2}. In degree
$3$ there are $10$ multiplicative generators and the corresponding relations
are given in Table \ref{tabledeg3}.

The code used to produce the relations is given (and commented) in the next
section.

\begin{exm}\label{exma} The monomial $b=X_{23}X_{12}$ is an element of the basis $B_4$ but is
  not part of the chosen weight $2$ multiplicative generators of Table
  \ref{tabledeg2}. Its dual 
  element in $V(\m_{0,5})$ is given by 
\[
b^*=[\om_{23}|\om_{12}]=[\om_{23}]\sha[\om_{12}]-[\om_{12}|\om_{23}].
\]

As previously, let $\gamma$ denote the path $p_{35} \circ p_{52} \circ
p_{24} \circ p_{41} \circ p_{13}$. Computing the iterated integral 
$\int_{\gamma} \Reg([\om_{23}|\om_{12}], \gamma)$, one finds
\begin{equation}
0=\int_{\gamma}\Reg( [\om_{23}|\om_{12}],\gamma)=-\zsha(X_1X_0)-\zsha(X_0X_1)+\zsha(X_1)^2. \label{exma0}
\end{equation}
In the other hand, the relations given by the iterated integrals of $[\om_{23}]$,
$[\om_{12}]$ and $[\om_{12}|\om_{23}]$ are (see Tables \ref{tabledeg1} and
\ref{tabledeg2}) 
\begin{equation}
0=\int_{\gamma} \Reg( [\om_{23}], \gamma)=2\left(\zsha(X_0)-\zsha(X_1)\right), \label{exma1} 
\end{equation} 
\begin{equation}
0=\int_{\gamma}\Reg( [\om_{23}],\gamma)=\zsha(X_0)-\zsha(X_1) \label{exma2}
\end{equation}
and
\begin{multline}
0=\int_{\gamma}\Reg( [\om_{12}|\om_{23}],\gamma)=2\zsha(X_0)^2-2 \zsha(X_1)
\zsha(X_0)\\+\zsha(X_1)^2- \zsha(X_0X_1)- \zsha(X_1X_0). \label{exma3}
\end{multline} 
Multiplying Equations \eqref{exma1} and \eqref{exma2} and subtracting Equation
\eqref{exma3}, one finds
\[
0=-2\zsha(X_0)\zsha(X_1)+\zsha(X_1)^2+\zsha(X_0X_1)+\zsha(X_1X_0).
\]
Using the shuffle relation on the product $\zsha(X_0)\zsha(X_1)$, one gets
\[
-\zsha(X_1X_0)-\zsha(X_0X_1)+\zsha(X_1)^2=0,
\]
which is exactly the relation given by the iterated integral $\int_{\gamma}
\Reg([\om_{23}|\om_{12}],\gamma)$ at Equation \eqref{exma0}. Here, we used the
shuffle relation on the term 
\[
-2\zsha(X_0)\zsha(X_1)
\] because this term  
corresponds to the following integrals
\[
 \int_{p_{35}}\Reg([\om_{23}] \sha [\om_{12}],p_{35})=
\left(\int_{p_{35}}\Reg([\om_{23}],p_{35})\right)
\left( \int_{p_{35}}\Reg([\om_{12}],p_{35})\right) 
\]
and 
\[
\int_{p_{24}}\Reg([\om_{23}] \sha [\om_{12}],p_{24})= 
\left(\int_{p_{24}}\Reg([\om_{23}],p_{24})\right)
\left( \int_{p_{24}}\Reg([\om_{12}],p_{24})\right).
\] 
\end{exm}

\begin{exm}
In weight $3$, let us consider the monomial $b=X_{24}X_{23}X_{12}$, which is an
element of the basis $B_4$, without being one of the chosen multiplicative
generators of Table \ref{tabledeg3}. Its dual element $b^*$ is given by
\[
b^*=[\om_{24}|\om_{23}|\om_{12}]+[\om_{23}|\om_{24}|\om_{12}]
+[\om_{23}|\om_{12}|\om_{24}]=[\om_{24}]\sha [\om_{23}|\om_{12}].
\]
The element $[\om_{23}|\om_{12}]$ in $V(\m_{0,5})$ is dual to the monomial
$X_{23}X_{12}$ which is not an element of the chosen weight $2$ multiplicative
generators (see Table \ref{tabledeg2}). However, we explained in the previous
example (Example \ref{exma}) how to derive the relation corresponding to
$X_{23}X_{12}$ from the relations corresponding to  $ X_{23}X_{12}$, $X_{23}$
and $X_{12}$.

As previously, let $\gamma$ denotes the path $p_{35} \circ p_{52} \circ
p_{24} \circ p_{41} \circ p_{13}$. The complete relation given by the iterated
integral $\int_{\gamma}\Reg(b^*,\gamma)$ is   
\begin{multline}\label{exam2}
-\zsha(X_1)^3 + 2 \zsha(X_1) \zsha(X_0X_ 1) - \zsha(X_0) \zsha(X_1X_ 0) \\
+ 2 \zsha(X_1) \zsha(X_1X_0) - 2 \zsha(X_0X_ 0X_1) - \zsha(X_0X_1X_0)=0.
\end{multline}
The relations given by the iterated integral of $[\om_{24}] $ and
$[\om_{23}|\om_{12}]$ are  respectively
\begin{equation} \label{exmb1}
\zsha(X_0)-\zsha(X_1)=0
\end{equation}
and 
\begin{equation}\label{exmb2}
-\zsha(X_1X_0)-\zsha(X_0X_1)+\zsha(X_1)^2=0.
\end{equation}
Multiplying those two equations one finds
\begin{multline}
-\zsha(X_1X_0)\zsha(X_0)-\zsha(X_0X_1)\zsha(X_0)+\zsha(X_1)^2\zsha(X_0) +\\
\zsha(X_1X_0)\zsha(X_1)+\zsha(X_0X_1)\zsha(X_1)-\zsha(X_1)^3=0.
\end{multline}
Now, using shuffle relations 
\[
-\zsha(X_0X_1)\zsha(X_0)=- 2 \zsha(X_0X_ 0X_1) - \zsha(X_0X_1X_0)
\]
and 
\[
\zsha(X_1)^2\zsha(X_0)=\zsha(X_1X_0)\zsha(X_1)+\zsha(X_0X_1)\zsha(X_1),
\]
one recovers the relation corresponding to $X_{24}X_{23}X_{12}$ given in
Equation \eqref{exam2}. As in the previous example, using the shuffle relations
between multiple zeta values for some products corresponds to the shuffle relation
between some products of iterated integrals.

One should also remark that it is possible to recover directly from Table
\ref{tabledeg3} the relation 
\[
- 2 \zsha(X_0X_ 0X_1) - \zsha(X_0X_1X_0)=0
\]
which is equivalent to Equation \eqref{exam2} as $\zsha(X_0)=\zsha(X_1)=0$. In
order to do so, one uses the relations given by the monomials $X_{34}X_{34}X_{45}$ and
$X_{24}X_{34}X_{45}$. 
\end{exm}
\begin{table}[h]
\begin{tabularx}{0.87\textwidth}{>{$\ds}c<{$}c>{$\ds}c<{$}c>{$\ds}l<{=0$}c}
\toprule
\multicolumn{1}{l}{\tetecoloneun}
&\hphantom{test}&\multicolumn{1}{p{2.5cm}}{Dual elements in $V(\m_{0,5})$}&
\hphantom{test}&\multicolumn{1}{l}{\tetecolonedeux } & \\
\toprule
 X_{1 2} &  &[\om_{1 2}] & &
\zetasha\left(X_0\right)-\zetasha\left(X_1\right) &
\findecoeff  \\ \midrule
 X_{2 3} &  &[\om_{2 3}]  & &
 2\left(\zetasha\left(X_0\right)-\zetasha\left(X_1\right)\right) &
\findecoeff  \\ \midrule
 X_{2 4} &  &[\om_{2 4}]  & &
\zetasha\left(X_0\right)-\zetasha\left(X_1\right) &
\findecoeff  \\ \midrule
 X_{3 4} &  & [\om_{3 4}]  & & 2
\left(\zetasha\left(X_0\right)-\zetasha\left(X_1\right)\right) &
\findecoeff  \\ \midrule
 X_{4 5} &  &[\om_{4 5}]  & &
\zetasha\left(X_0\right)-\zetasha\left(X_1\right) &
\findecoeff\\ 
\bottomrule \\
\end{tabularx}

\caption{Explicit set of relations equivalent to \eqref{III} in degree
  $1$} 
\label{tabledeg1}
\end{table}
\begin{table}[h]
\begin{tabularx}{\linewidth}{>{$\ds}c<{$}c>{$\ds}l<{$}c>{$\ds}l<{$}c}
\toprule
\multicolumn{1}{l}{\tetecoloneun}
& &\multicolumn{1}{l}{Dual elements in $V(\m_{0,5})$}&
&\multicolumn{1}{l}{\tetecolonedeux } & \\
b_4 \in B_4 & &b_4^*=\sum l_{b_4,W}\om_W & & & \\
\toprule
\toprule
 X_{2 4} X_{4 5} &  &-[\om_{1  2} |\om_{2  4}] + 
 [\om_{ 2  4}|\om_{4  5}] &  &
\zetasha\left(X_0\right)
\zetasha\left(X_1\right)-\zetasha\left(X_1\right){}^2=0 &
\findecoeff  \\ \midrule
X_{2 4} X_{3 4} &  & -[\om_{1  2}|\om_{2  4}] + 
 [\om_{ 2  3}|\om_{ 2  4}] 
&  &
-\zetasha\left(X_0\right){}^2+\zetasha\left(X_1\right)
\zetasha\left(X_0\right)-& \findecoeff  \\ 
&  &\hphantom{te}~ - [\om_{2  3}|\om_{3  4}] + 
[\om_{ 2  4}|\om_{3  4}]&  &\multicolumn{1}{r}{$\ds 2
\zetasha\left(X_1\right){}^2+ \zetasha\left(X_0X_0\right)+ 
$}\\
&  & &  &\multicolumn{1}{r}{$\ds \zetasha\left(X_0X_1\right)+ 
\zetasha\left(X_1X_0\right)+ 
$}\\ 
&  & &  &\multicolumn{1}{r}{$\ds \zetasha\left(X_1X_1\right)=0 $}\\
 \midrule
%
X_{3 4} X_{4 5} &  &  [\om_{ 3  4}|\om_{4  5}]&  & 
2\zetasha\left(X_0\right){}^2-\zetasha\left(X_1\right)
\zetasha\left(X_0\right)-& \findecoeff  \\ 
&  & &  &\multicolumn{1}{r}{$\ds \zetasha\left(X_0X_1\right)-\zetasha
\left(X_1X_0\right) =0 $}\\ \midrule
 X_{1 2} X_{2 3} &  &  [\om_{ 1  2}|\om_{2  3}]&  & 
2\zetasha\left(X_0\right){}^2-2 \zetasha\left(X_1\right)
\zetasha\left(X_0\right)+
&
\findecoeff  \\ 
&  & &  &\multicolumn{1}{r}{$
\ds \zetasha\left(X_1\right){}^2- \zetasha\left(X_0X_1\right)-$}\\
&  & &  &\multicolumn{1}{r}{$
\ds \zetasha\left(X_1X_0\right)  =0$}\\ \midrule
\bottomrule\\
\end{tabularx}

\caption{Explicit set of relations equivalent to \eqref{III} in degree
  $2$} 
\label{tabledeg2}
\end{table}
\begin{table}[ht]
\begin{tabularx}{\linewidth}{>{$\ds}c<{$}c>{$\ds}l<{=0$}c}
\toprule
\multicolumn{1}{l}{\tetecoloneun}
&\hphantom{test}&\multicolumn{1}{l}{\tetecolonedeux } & \\
\toprule
 {X_{34} X_{24} X_{24}} &  &{
 - \zetasha\left(X_0 X_0 X_1\right)- \zetasha\left(X_0 X_1 X_0 
\right)- \zetasha\left(X_1 X_0 X_0\right)} & \findecoeff  \\ \midrule
 {X_{12} X_{23} X_{23} }&  &{
  \zetasha\left(X_0 X_1 X_1\right)- \zetasha\left(X_1 X_0 X_0
 \right)}& \findecoeff  \\ \midrule
 {X_{34} X_{45} X_{45}} &  &
 { \zetasha\left(X_0 X_1 X_1\right)- \zetasha\left(X_1 X_0 X_0
 \right)}& \findecoeff  \\ \midrule
 {X_{45} X_{24} X_{24}} &  &{
  \zetasha\left(X_0 X_1 X_1\right)- \zetasha\left(X_1 X_0 X_0
 \right)}& \findecoeff  \\ \midrule
 { X_{12} X_{12} X_{23} }&  &
  { \zetasha\left(X_1 X_1 X_0\right)- \zetasha\left(X_0 X_0 X_1
 \right)}& \findecoeff  \\ \midrule
 {X_{34} X_{34} X_{45}} &  &{
  \zetasha\left(X_1 X_1 X_0\right)- \zetasha\left(X_0 X_0 X_1
 \right)}& \findecoeff  \\ \midrule
 {X_{24} X_{45} X_{45}} &  &{
  \zetasha\left(X_1 X_1 X_0\right)- \zetasha\left(X_1 X_0 X_0
 \right)}& \findecoeff  \\ \midrule
{X_{24} X_{34} X_{34}} &  &\multicolumn{1}{l}{$\ds
 { \zetasha\left(X_0 X_0 X_1\right)+ \zetasha\left(X_0 X_1 X_0
 \right)- \zetasha\left(X_0 X_1 X_1\right)} $}& \\
 & & \multicolumn{1}{r}{$\ds{ + \zetasha\left(X_1 X_0 
 X_0\right)+ \zetasha\left(X_1 X_1 X_0\right)=0}$} &
\findecoeff  \\ \midrule
 {X_{24} X_{45} X_{34} }&  &{
  \zetasha\left(X_1 X_0 X_0\right)+ \zetasha\left(X_1 X_0 X_1
 \right)+ \zetasha\left(X_1 X_1 X_0\right)} & \findecoeff  \\ \midrule
 {X_{24} X_{34} X_{45}} &  &{
  \zetasha\left(X_0 X_1 X_0\right)+2
  \zetasha\left(X_1 X_1 X_0\right) }& \findecoeff  \\
\bottomrule\\
\end{tabularx}


\caption{Explicit set of relations equivalent to \eqref{III} in degree
  $3$ where we already have used the relations $\protect\zsha(X_0^k)=\protect\zsha(X_1^k)=0$.
} 
\label{tabledeg3}
\end{table}
\begin{table}[ht]
\begin{tabularx}{\linewidth}{>{$\ds}c<{$}c>{$\ds}l<{$}}
\toprule
\multicolumn{1}{l}{\tetecoloneun}&\hphantom{test}&
\multicolumn{1}{l}{Dual elements in $V(\m_{0,5})$ } \\
b_4 \in B_4 & &b_4^*=\sum l_{b_4,W}\om_W \\
\toprule
X_{34}X_{24}X_{24} & &[\om_{12}|\om_{24}|\om_{24}] +[\om_{23}|\om_{12}|\om_{24}]
- [\om_{23}|\om_{23}|\om_{24}] + [\om_{23}|\om_{23}|\om_{34}] \\
 & & \multicolumn{1}{r}{$- [\om_{23}|\om_{24}|\om_{24}] 
+ [\om_{23}|\om_{34}|\om_{24}] + [\om_{34}|\om_{24}|\om_{24}]$} \\
\midrule
X_{12}X_{23}X_{23} & & [\om_{12}|\om_{23}|\om_{23}] \\
\midrule
 X_{34}X_{45}X_{45} & &[\om_{34}|\om_{45}|\om_{45}] \\
\midrule
 X_{45}X_{24}X_{24} & &[\om_{12}|\om_{24}|\om_{24}]+[\om_{45}|\om_{24}|\om_{24}]
\\
\midrule
 X_{12}X_{12}X_{23}& &[\om_{12}|\om_{12}|\om_{23}]\\
\midrule
 X_{34}X_{34}X_{45}& &[\om_{34}|\om_{34}|\om_{45}] \\
\midrule
 X_{24}X_{45}X_{45}& &[\om_{12}|\om_{12}|\om_{24}]-[\om_{12}|\om_{24}|\om_{45}]
 +[\om_{24}|\om_{45}|\om_{45}]\\
\midrule
 X_{24}X_{34}X_{34}& &[\om_{12}|\om_{12}|\om_{24}]-[\om_{12}|\om_{23}|\om_{24}]
+[\om_{12}|\om_{23}|\om_{34}]-[\om_{12}|\om_{24}|\om_{34}]\\
& &\multicolumn{1}{r}{$
-[\om_{23}|\om_{12}|\om_{24}]+[\om_{23}|\om_{23}|\om_{24}]
-[\om_{23}|\om_{23}|\om_{34}]+[\om_{23}|\om_{24}|\om_{34}]$}\\
& &\multicolumn{1}{r}{$
-[\om_{23}|\om_{34}|\om_{34}]+[\om_{24}|\om_{34}|\om_{34}]$}\\
\midrule
X_{24}X_{45}X_{34} & &[\om_{12}|\om_{12}|\om_{24}]-[\om_{12}|\om_{23}|\om_{24}]
+[\om_{12}|\om_{23}|\om_{34}]-[\om_{12}|\om_{24}|\om_{34}]\\
& & \multicolumn{1}{r}{$+[\om_{24}|\om_{45}|\om_{34}]$}\\
\midrule
X_{24}X_{34}X_{45}& &[\om_{12}|\om_{12}|\om_{24}]-[\om_{12}|\om_{24}|\om_{45}]
-[\om_{23}|\om_{12}|\om_{24}]+[\om_{23}|\om_{24}|\om_{45}] \\
& &\multicolumn{1}{r}{$
-[\om_{23}|\om_{34}|\om_{45}]+[\om_{24}|\om_{34}|\om_{45}]$}\\
\bottomrule\\
\end{tabularx}

\caption{Correspondence between ten multiplicative generators of weight $3$ in
  $\UB$ and 
  their dual elements in $V(\m_{0,5})$} 
\label{tabledeg3dual}
\end{table}

%
\clearpage
\section{Appendix : algorithm}\label{annexealgo}
\subsection{Comments}The above computations were done using the software
Mathematica because its replacement rules and pattern recognition are very
efficient dealing with words. In this section, the algorithms used to produce the tables from the
previous sections are commented.

The naive algorithms described below were originally intended to
provide help in guessing the family 
of relations \eqref{IIIMZV} given by the pentagon relation. Concentrating our
attention on understanding \eqref{IIIMZV}, proving it and explaining the connection with 
the bar construction on $\m_{0,5}$, the author did not make a particular effort to
improve the algorithms (and their results).

\subsection{Law, relations, and basis} Using Mathematica, we need to define a
new \emph{NonCommutativeMultiply} function which behaves like the desired
multiplicative law for a polynomial algebra with
non-commutative variables. This is done using Mathematica's elementary
operations such as pattern recognition and replacement rules. All the
non-commutative products used in the algorithms below are understood as this new
\emph{NonCommutativeMultiply} function.

In order to write words in $\{X_{12}, X_{23},X_{34},X_{45},X_{51}\}^*$ in the
basis $B_4$, we need to use the relations in $\UB$ and thus to implement the
functions $REl51$ and $Relcom$.
\begin{itemize}
\item The function $Rel51$ writes the letter $X_{51}$ in terms of $X_{23}$, $X_{24}$, $X_{34}$:
\begin{align*}
Rel51 : \quad & X_{51} \longmapsto X_{23}+X_{24}+X_{34}
\end{align*}
\item The function $Relcom$ uses the commutation relations to write a product
  $X_{ij}X_{kl}$ with $X_{12}$ or $X_{23}$ on the right side. It does nothing to
  the product $X_{ij}X_{kl}$ if it is a word in the letters $X_{12}$, $X_{23}$ or if it
  is a word in the letters $X_{34}$, $X_{45}$ and $X_{24}$. Beginning with a
  word in $\axb$ and iterating applications of the function $Relcom$, one obtains
  its decomposition in the basis $B_4$.
\begin{align*}
Relcom :\quad & X_{12}X_{kl} \longmapsto X_{k l}X_{1 2} \quad \mx{for }k=3 \mx{ and
} l=4,\mx{ or } k=4, \mx{ and }l=5 \\
 &X_{23}X_{45} \longmapsto X_{45}X_{23}\\
&X_{1 2}X_{2 4}\longmapsto (X_{2 4}+X_{3 4}+X_{4 5})X_{2 4} -  
         X_{2 4}  (X_{2 4} + X_{3 4} + X_{4 5}) \\
&\hphantom{X_{1 2}X_{2 4}\longmapsto (X_{2 4}+X_{3 4}+X_{4 5})X_{2 4} -  
         X_{2 4}  (X_{2 4}+X}+ X_{2 4}  X_{12} \\
&X_{2 3}  X_{2 4} \longmapsto
          X_{2 4}  X_{3 4} - X_{3 4}  X_{2 4} + X_{2 4}  X_{2 3} \\
&X_{2 3}  X_{3 4} \longmapsto
     X_{3 4}  X_{2 4} - X_{2 4}  X_{3 4} + X_{3 4}  X_{2 3}
\end{align*}
\end{itemize}
Computing up to a fixed weight $n$, we consider a basis restricted to
weight $n$ and less, and we define functions $BX0X1$ and $B_4$ which give respectively 
the list of the corresponding monomials.
\begin{itemize}
\item $BX0X1(n):=\mx{List of words } W \in \ax \mx{ with } |W|\leqs n$.
\item $B_4(n):=\mx{List of words } W=W_1W_2$  with $ W_1\in \axq$, $ W_2\in
  \axud$ and $|W|\leqs n$.
\end{itemize}
Then, for any given $A$ in $\UB$ given as
\[A=\sum_{\substack{W\in \{X_{51},X_{34},X_{45},X_{12},X_{23}\}^*, \\ |W|\leqs n}} a_W W\]
one can write $A$ in the basis $B_4$ by using the function $DecB_4$ below:
\begin{itemize}
\item $DecB_4 :=$
\begin{itemize}
\item $A_1:=Rel51(A)$ and expand $A_1$ as $\sum_{W\in \axb \, |W|\leqs n} b_W W$.
\item Do $A_1:=Relcom(A_1)$ until $A_1=\sum_{W\in B_4 \, |W|\leqs n} c_W W$.
\end{itemize}
\end{itemize}
This function is defined using the build-in function \emph{Expand} and
\emph{Collect} together with the previously defined functions. 
 For later use, we need a function $Deg(A,n)$ that truncates $A$
at weight $n$. 
\subsection{Exponential, associator}
Working up to a fixed weight $n$, we now construct a function
that takes two variables $A$ and $B$ and an integer $n$ as inputs and 
gives as output a general polynomial $\Phi_n(A,B)$ of degree $n$ with formal coefficients
\[
\Phi_n(A,B)=1+\sum_{\substack{W \in \ax, \, \{A,B\}^*\neq \emptyset \\|W|\leqs n}}
(-1)^{\de(W)}Z_{\bar W}W,
\]
where $\bar W$ is obtained from $W$ by sending $A$ to $X_0$ and $B$ to $X_1$.

We also define a non-commutative exponential up to degree $n$
\[
Exp_n(A)=\sum_{0\leqs k\leqs n } \frac{A^k}{k !}.
\]
\subsection{Development of the associator relations}
We detail here how we develop the hexagonal and pentagonal relations. 

In order to develop the hexagonal relation
\[
e^{p*X_0}\Phi_n(X_{\infty},X_0)e^{p*X_{\infty}}
\Phi_n(X_1,X_{\infty})e^{p*X_1}\Phi_n(X_0,X_1)
\]
truncated in degree $n$ and expand in the basis given by the words in $X_0$ and
$X_1$. We proceed as follows:
\begin{enumerate}
\item We compute the successive products keeping only
the terms of weight less or equal to $n$. That is, we compute
\begin{align*} P_1&=Deg(e^{pX_0}\Phi_n(X_{\infty},X_0), n), \\
   P_2&=Deg(P_1e^{pX_{\infty}},n), \\ 
\ldots  & \\
P_6&=Deg(P_5\Phi_n(X_0,X_1),n)
\end{align*}
\item Then, we apply $X_{\infty} \longmapsto -X_0-X_1$ and $p \longmapsto i\pi$ to $P_6$.
\item Finally, we expand the expression
and collect the terms of the sum with respect to the list $BX0X1(n)$ and obtain an
expression
\[
\sum_{W \in \ax, \, |W|\leqs n}a_W W.
\]
\end{enumerate}
The coefficients $a_W$ are expressed as a sum of  products of a rational
coefficient, a power of $i\pi$ and a product of $Z_{U}$ for $U$
in $\ax$. Formally replacing $Z_U$ by
$\zsha(U)$, the set of relations \eqref{IIMZV} is given by 
\[
a_W=0 \qquad (W\neq \emptyset).
\]

Similarly, in order to find the set of relations \eqref{IIIMZV} arising from the
$5$-cycle equation \eqref{III},  we expand the product
\[
Penta=\Phi_n(X_{12},X_{23}) \Phi_n(X_{34},X_{45}) \Phi_n(X_{51},X_{12})
\Phi_n(X_{23},X_{34}) \Phi_n(X_{45},X_{51}),
\]
computing the successive products and keeping only the part of
weight less or equal to $n$ at each step. 

Then, we develop the corresponding expression with the variables $X_{12}$, 
$X_{23}$, $X_{34}$, $X_{45}$, $X_{51}$
in the basis $B_4$, applying the function $DecB_4$ to the expression $Penta$, to
obtain an expression of the form
\[
\sum_{b \in B_4, \, |b|\leqs n}a'_b b.
\]
The coefficients $a'_b$ are a sum of products of $Z_U$ for $U$
in $\ax$. One can formally replace $Z_U$ by $\zsha(U)$ and obtain the set of
relations \eqref{IIIMZV} setting $a'_b=0$ for $b$ not equal to $1$.

\subsection{Using for \eqref{III} the equivalent set of relations given in
  \eqref{IIIMZV}} 
We describe here how to obtain the family of relations \eqref{IIIMZV} up to
degree $n$, that is:
\begin{align*}
\intertext{For any $ b \in B_4$ with $ |b|\leqs n$, $b \neq 1$} \\
\sum_{W \in \axb} l_{b,W}C_{5,W}&=0,
\end{align*}
by first generating the coefficients $C_{5,W}$ and then the coefficient
$l_{b,W}$.
 
In order to compute the coefficients $C_{5,W}$ for any word $W$ in
$\axb$,  we first construct a function $Decw$ that takes a word as input and gives as output
all the possibilities to cut it into five sub-words.
\[
Decw(W):=\mx{List of decomposition }(U_1,\ldots, U_5) \mx{ such that } U_1 \cdots U_5=W.
\]
The function $Decw$ is built inductively by first giving the list of all
decompositions $U_1U_2=W$, then iterating the process on each $U_1$ and so forth.

Then, we implement functions corresponding to the $\pg$ (Definition \ref{pgdef})
by programming the 
behavior on the letters as follow
\begin{align*}
rho(i,X_{12}):= &X_0 \mx{ if } i=1,\, X_1 \mx{ if } i=3 \mx{ and }0  \mx{
  otherwise, } \\
rho(i_,X_{23}):= &X_0 \mx{ if } i=2,3,\, X_1 \mx{ if } i=1,5 \mx{ and }0  \mx{
  otherwise, }\\
 rho(i_,X_{45}):= &X_0 \mx{ if } i=5,\, X_1 \mx{ if } i=2 \mx{ and }0  \mx{
  otherwise, }\\
 rho(i_,X_{34}):= &X_0 \mx{ if } i=2,3\, X_1 \mx{ if } i=4,5 \mx{ and }0  \mx{
  otherwise, }\\
 rho(i_,X_{24}):= &X_0 \mx{ if } i=3,\, X_1 \mx{ if } i=5 \mx{ and }0  \mx{
  otherwise, }\\
\end{align*}
which extends to words. The function $Zrho$ takes as input $i$ and a
word $U_i$ in $\axb$  and gives the coefficient
\[
(-1)^{\de(\pg(U_i))}\zsha({\pg(U_i)}).
\]
\begin{itemize}
\item $Zrho(i, U_i):=$
\begin{itemize}
\item  Do $V=rho(i,U_i)$ and $s=\de(V)$
\item output : $(-1)^{s}\zsha(V)$
\end{itemize}
\end{itemize}

Now, from a decomposition  
\[
 U_1\cdots U_5=W
\]
we can recover the coefficient
\begin{multline*}
(-1)^{\dea(U_1)+\deb(U_2) +\dec(U_3)+ \ded(U_4)+ \dee(U_5)} \\
\zsha(\pa(U_1))
\zsha(\pb(U_2))
\zsha(\pc(U_3)
\zsha(\pd(U_4))
\zsha(\pe(U_5)),
\end{multline*}
that is 
\[
Z(U_1,U_2,U_3,U_4,U_5):= \prod_{i=1}^5 Zrho(i,U_i).
\]

Using functions $Decw$ and $Z$, we now compute the sum over the whole set of
decompositions and obtain a function that 
gives the coefficient $C_{5,W}$:
\[
C5(W):=\sum_{(U_1,\ldots, U_5) \in Decw(W)}Z(U_1,U_2,U_3,U_4,U_5).
\] 


We now compute the $l_{b,W}$ coefficients up to some weight by the following algorithm:
\begin{itemize}
\item Begin with $L:= \mx{List of words }W \in
  \axb$, $|W|\leqs n$ .
\item $L1:=$ for each element in $L$ apply $DecB_4$
\item $L2:=$ for each element in $L1$ replace $\sum_{b \in B_4} l_{b, W} b$ by
  the list of the corresponding $l_{b,W}$.
\item output : $L2$.
\end{itemize}
One can then compute for any $b \in B_4$ with $|b|\leqs n$, $b\neq 1$ 
\[
\sum_{W \in \axb} l_{b,W}C_{5,W}
\]
which is the L.H.S. of  \eqref{IIIMZV}.
\begin{rem}
One could imitate the algorithm that gives $C_{5,W}$ in order to recover the
pentagon relation using the bar construction side of the story. The decomposition
function $Decw $ could be directly reused to cut a bar symbol
$\om_W$ in five pieces. The $rho$ function corresponds to the
implementation of the regularization $\on{Reg}$ on the $u_{ij}$. In order to
recover the pentagon relation from 
\[
\forall b^* \in B^* \qquad \iti_{\gamma}\Reg(b^*,\gamma)=0,
\] 
$B^*$ being a basis of $V(\m_{0,5}$), one will have to implement linearity and
the correspondence between formal bar symbols and their iterated integrals. The
latter should be similar to the function $Zrho$ but one may need to
be careful with possible signs.
\end{rem}
\subsection{Remarks}
The author, having recently discovered the software Sagemath, thinks that it may be
easier to do the computations with Sagemath. This is because Sagemath seems
to work well with non-commutative formal power series and it has large libraries to deal
with words. 

In \cite{mapplebigotte}, M. Bigotte and N.E. Oussous have described a Maple
package to work with non-commutative power series. However, it was not yet
possible to have access to this package when this work began.



\bibliographystyle{amsalpha}
\nocite{}
\bibliography{pentabib}


\end{document}